\documentclass[reqno]{amsart} 

\usepackage{amsfonts, amsmath, amsthm, amssymb, latexsym, graphicx, geometry, xcolor, colortbl,  mathtools,  mathrsfs}

\theoremstyle{plain}
\newtheorem{theorem}{Theorem}[section]
\newtheorem{corollary}[theorem]{Corollary}
\newtheorem{lemma}[theorem]{Lemma}
\newtheorem{proposition}[theorem]{Proposition}
\newtheorem{def-theorem}[theorem]{Definition and Theorem}
\theoremstyle{definition}
\newtheorem{definition}[theorem]{Definition}

\newtheorem{remark}[theorem]{Remark}

\newtheorem{example}[theorem]{Example}

\numberwithin{equation}{section}

\newcommand{\R}{\mathbb{R}}

\newcommand{\N}{\mathbb{N}}
\newcommand{\Z}{\mathbb{Z}}
\newcommand{\ra}{\rightarrow}
\newcommand{\xra}{\xrightarrow}

\newcommand{\F}{\mathbb{F}}

\title[Freezing limits for Bessel functions and Bessel processes]{Some freezing limits for Bessel functions and Bessel processes with drift of type $B_N$}

\author{Jan Richter, Michael Voit} 
\address{Fakult\"at Mathematik, Technische Universit\"at Dortmund,
          Vogelpothsweg 87,
          D-44221 Dortmund, Germany}
\email{jrichter@mathematik.tu-dortmund.de, michael.voit@math.tu-dortmund.de}


\begin{document}
\date{\today}

\begin{abstract}
  We use the series representation of the Bessel functions of type $B_N$ in terms of Jack polynomials and show that
  $$\lim_{k_1\to\infty} J_{(k_1,k_2)}^B(x,k_1y)^{1/k_1}= 2^N\prod\limits_{i=1}^N
  \Bigg( e^{\sqrt{1+x_i^2 y_i^2}-1}\cdot \frac{1}{\sqrt{1+x_i^2 y_i^2}+1} \Bigg)$$
  for  $x,y\in\mathbb R^N$ and $k_2\ge0$.
  Moreover, the known Laplace-type integral representations for $N\ge1$ and $k_2=0,1/2,1,2$
and for $N=2$ and $k_2>0$ by R\"osler and Demni respectively lead to related limits for
$$\lim_{k_1 \to \infty} \partial_{x_j} J_{(k_1,k_2)}^B(x,k_1y) /(k_1 \cdot J_{(k_1,k_2)}^B(x,k_1 y)) \quad (j=1,\ldots,N).$$
  These limits lead to weak limit results for the associated Bessel processes with drift.
  For $k_2=1/2,1,2$, these limit results have applications to radial parts of Brownian motions with drift on the
  $M\times N$-dimensional matrices over $\mathbb R,\mathbb C$, and the quaternions for $M\to\infty$.
  We also discuss these limits in the Dunkl case $N=1$.
\end{abstract}

\maketitle

\section{Introduction}

Several weak freezing limit theorems for Bessel processes on closed Weyl chambers $C_N\subset\mathbb R^N$
associated with the root system $A_{N-1}$ and $B_N$ were given in \cite{AKM1, AKM2, AM, AnV2} where these results partially
are based on limit results for the associated Bessel functions. The notion ``freezing limit'' here means that limits for
large  parameters $k\to\infty$ are considered where these parameters have the interpretation
of inverse temperatures for the particle systems described by these Bessel processes.
Further related freezing limit results like central limit theorems can be found
in \cite{DE, V1, AHV, GK}, where in all papers above the starting points are
fixed points $x\in C_N$. These multivariate Bessel processes
 are closely related to classical random matrix models;
 in particular the Bessel processes of type $A_{N-1}$ are the well-known Dyson Brownian motions;
 see e.g. \cite{AGZ, Kat, F}.
Furthermore, for starting points of the form $k\cdot x$ with $x\in C_N\setminus\{0\}$,
 freezing results are given in \cite{AnV1, VW} by using the stochastic differential equations of the Bessel processes.
One might ask whether it is possible to derive such limit theorems directly from the known transition probabilities
of the Bessel processes in an analytic way by using limits for the Bessel functions.
This demand on freezing-type limit results for Bessel functions
is supported by the fact that, besides the usual Bessel processes
on $C_N$, there exists the notion of  Bessel processes with drift vectors on $C_N$ in \cite{V2}.
This notion is quite natural, as for instance the usual radial parts of  Brownian motions on $\mathbb R^d$ with arbitrary
drift vectors are examples of one-dimensional Bessel processes with drift; see e.g. \cite{PY}. Moreover,
the ordered spectra of Wishart processes with drift on the $N$-dimensional positive semidefinite matrices
as in \cite{DDMY} are Bessel processes with drift of type $B_N$ on the Weyl chamber
$$C_N^B:=\{x\in \mathbb R^N: \quad x_1\ge x_2\ge\ldots\ge x_N\ge 0 \},$$
and the ordered spectra of
Brownian motions with arbitrary drift vectors on the vector spaces of  $N$-dimensional Hermitian matrices are
Bessel processes with drift of type $A_{N-1}$; see \cite{AuV, V2}.
In the following we refer to \cite{R2, RV} for Bessel processes associated with  root systems without drift,
to \cite{CGY} for stochastic analysis of these processes, and
to \cite{D, R3, A} for
the underlying Bessel functions. We recapitulate the following definition from \cite{V2}.

\begin{definition}\label{def-bessel-p-drift}
  Let $R$ be a root system on $\mathbb R^N$, $R_+$ some subset of positive roots, $W$ the associated reflection group,
  and $C_N$ the closed Weyl chamber associated with $R_+$. Let $ k\ge0$ be some multiplicity function, and
 let $w\in C_N$ be a ``drift vector''. Then the operator
   \begin{equation}\label{Bessel-laplace-drift-gen}
  L_k^w f(x):= \frac{1}{2}\Delta f(x)  + \sum_{\alpha\in R_+}
k(\alpha)
\frac{\langle\nabla f(x),\alpha\rangle}{\langle\alpha ,x\rangle}+\frac{\langle \nabla_x J_k(x,w), \nabla f(x)\rangle }{J_k(x,w)}
   \end{equation}
   for $W$-invariant $f\in C^{(2)}(\mathbb R^N)$ is the generator of a Feller diffusion on $C_N$ with reflecting boundaries.
   These diffusions have the transition densities
   \begin{equation}\label{density-transition-bessel-drift-gen}
  K_t^{w,k}(x,A)=\frac{|W|\cdot c_k\cdot e^{-\|w\|^2 t/2}}{t^{\gamma_k+N/2}} \int_A e^{-(\|x\|^2+\|y\|^2)/(2t)}
 \frac{ J_k(\frac{x}{\sqrt{t}}, \frac{y}{\sqrt{t}}) \cdot J_k(y,w)}{J_k(x,w)}
\cdot w_k(y)\> dy
\end{equation}
   for $t>0$, $x\in C_N$ and  Borel sets $A\subset  C_N$ with the weight function
   $w_k(y) := \prod_{\alpha\in R_+} |\langle\alpha,y\rangle|^{2k(\alpha)}$,
    the Bessel function $J_k$ associated with the root system $R$ and multiplicity $k$, and with suitable constants $c_k,\gamma_k>0$.
These diffusions  have modifications with continuous paths. They
will be called Bessel processes of type $R$ with multiplicity $k$ and drift $w$.   
\end{definition}

Eqs.~\eqref{Bessel-laplace-drift-gen} and \eqref{density-transition-bessel-drift-gen}
indicate that freezing limits for
these Bessel processes with drift might be derived from
freezing limits of the form
\begin{equation}\label{general-limit-functions}
  \lim_{\kappa\to\infty}\frac{\nabla_x J_{k(\kappa)}(x,\kappa y) }{\kappa \cdot J_{k(\kappa)}(x,\kappa y)}
\quad\text{and}\quad \lim_{\kappa\to\infty} J_{k(\kappa)}(x,\kappa y)^{1/\kappa}\end{equation}
for the Bessel functions  for  $x,y\in C_N$ and  multiplicities $k(\kappa)$  where all or some components of $k$ depend
linearly on $\kappa$.
These limits \eqref{general-limit-functions} are particularly interesting for the root systems $A_{N-1}$ with $k=\kappa$ as well as for
the root systems
$$B_N=\{\pm e_i:\> i=1,\dots,N\}\cup\{\pm e_i\pm e_j:\> i,j=1,\dots,N; \> i<j\}$$
 with $k=(k_1,k_2)=(\kappa,\kappa\nu)$  or  $k=(k_1,k_2)=(\kappa,k_2)$
 with constant $\nu,k_2\ge0$ where the multiplicities $k_1,k_2$ are associated with the sets 
$\{\pm e_i:\> i=1,\dots,N\}$ and $\{\pm e_i\pm e_j:\> i,j=1,\dots,N; \> i<j\}$.

Unfortunately, we are unable to compute these limits for $A_{N}$ and $B_N$ 
 for $N\ge2$ in general. For this reason, we here restrict our attention  to the root system $B_N$ with $k=(k_1,k_2)$,
 $k_2\ge0$ fixed, $k_1\to\infty$. The main result of this paper  is as follows:

\begin{theorem}\label{limit-BN-functions-intro}
Let  $N \in \N$ and $k_2\ge 0$. Then, locally uniformly for $x,y \in C_N^B$, 
\begin{align}\label{limit-BN-functions-intro1}
\lim_{k_1 \to \infty} \Bigg( J_{(k_1,k_2)}^{B}(x,k_1 y) \Bigg)^{1/k_1}
&= 2^N\prod\limits_{i=1}^N \Bigg( e^{\sqrt{1+x_i^2 y_i^2}-1}\cdot \frac{\sqrt{1+x_i^2 y_i^2}-1}{x_i^2 y_i^2} \Bigg)\notag\\
&= 2^N\prod\limits_{i=1}^N \Bigg( e^{\sqrt{1+x_i^2 y_i^2}-1}\cdot \frac{1}{\sqrt{1+x_i^2 y_i^2}+1} \Bigg).
\end{align}
Moreover, if  $N=1,2$ or if $N\ge3$ with $k_2=0,1/2,1,2$, then for all  $j=1,\ldots,N$, and locally uniformly for $x,y$ in the interior of $C_N^B$,
\begin{equation}\label{limit-BN-functions-intro2}
\lim_{k_1 \to \infty} \frac{\partial_{x_j} J_{(k_1,k_2)}^B(x,k_1 y) }{k_1 \cdot J_{(k_1,k_2)}^B( x,k_1 y)}
= \frac{\sqrt{x_j^2y_j^2+1}-1}{x_j}.
\end{equation}
\end{theorem}

Notice that \eqref{limit-BN-functions-intro2} formally is the logarithmic derivative of \eqref{limit-BN-functions-intro1}.
We expect that
\eqref{limit-BN-functions-intro2}  holds for all $N$ and $k_2\ge0$. We shall prove 
\eqref{limit-BN-functions-intro1} and
\eqref{limit-BN-functions-intro2}  by different methods. 
For  $k_2=1/2,1,2$ and general $N$ we use
the explicit Laplace integral representation  for $J_{(k_1,k_2)}^{B}$ in \cite{R4} which is based on the fact that here for certain discrete
$k_1$, the Bessel functions 
$J_{(k_1,k_2)}^B$ are spherical functions; see \cite{FK}. Moreover,  the case $J_{(k_1,0)}^{B_N}$ can be reduced to the case $N=1$.
Furthermore, for $N=2$, the Laplace integral representation for $J_{(k_1,k_2)}^{B_2}$ of Demni (see \cite{De, AD})
together with 
 \eqref{limit-BN-functions-intro1} and
\eqref{limit-BN-functions-intro2} for $N=1$ imply that for $\nu_1,\nu_2\ge0$ and $x,y\in C_2^B$ the limits
\begin{equation}\label{limit-BN-functions-intro3}
\lim_{\kappa \to \infty} \Bigg( J_{(\kappa \nu_1,\kappa \nu_2)}^{B_2}(x,\kappa y) \Bigg)^{1/\kappa} \quad\quad\text{and}\quad\quad
\lim_{\kappa \to \infty} \frac{\partial_{x_j} J_{(\kappa \nu_1,\kappa \nu_2)}^{B_2}(x,\kappa y) }{\kappa \cdot J_{(\kappa \nu_1,\kappa \nu_2)}^{B_2}(x,\kappa y)}  \quad(j=1,2)
\end{equation}
exist where these limits cannot be computed usually, as one would have to solve some complicated algebraic equations.
Moreover, the limits in \eqref{limit-BN-functions-intro3} hold locally uniformly in $x$ and $y$.
However, this approach and \eqref{limit-BN-functions-intro1} lead to \eqref{limit-BN-functions-intro2}  for $N=2$
and general $k_2>0$. For the details  see Section 2.

On the other hand, for arbitrary  $k_2>0$ and $N\ge3$
we  employ the series expansion of $J_{(k_1,k_2)}^B$ in terms of Jack polynomials (see \cite{Kan, FK, R4})
in order to derive \eqref{limit-BN-functions-intro1}.
This  proof of \eqref{limit-BN-functions-intro1} via Jack polynomials is more involved and is based on some limit result for Jack polynomials
which we  discuss in Section 3. Section 4 is then devoted to the proof of \eqref{limit-BN-functions-intro1} for general $k_2>0$.
Unfortunately, we are not able to extend the series approach also to the limit \eqref{limit-BN-functions-intro2}.

The limit \eqref{limit-BN-functions-intro1} should be compared with recent
uniform bounds for Bessel functions and Dunkl kernels for general root systems in \cite{L} and \cite{CLWW}.

The limit result 
\eqref{limit-BN-functions-intro1} will be used  in order to derive a weak limit theorem for the Bessel processes
$(X^{B,w}_{t,(k_1,k_2),x})_{t \geq 0}$ of type $B_N$ with multiplicity $(k_1,k_2)$,  drift $w \in C^N_B$, and starting point $x\in C^N_B$
for $k_1\to\infty$ and $k_2\ge0$ fixed. In order to motivate the limit, let us assume for a moment that
\eqref{limit-BN-functions-intro2} is proved for arbitrary $k_2\ge0$. Consider the renormalized
generators $\widetilde L_{(k_1,k_2)}^{B,w}:= \frac{1}{k_1} L_{(k_1,k_2)}^{B,w}$  with  $L_{(k_1,k_2)}^{B,w}$ the generator of a Bessel process of type $B_N$
with drift $w$ as  in Definition \ref{Bessel-laplace-drift-gen}. Then,
\begin{align}
\widetilde L_{(k_1,k_2)}^{B,w}f(x)&=\frac{1}{2k_1} \Delta f(x) + \frac{k_2}{k_1} \cdot \sum\limits_{i,j=1,\ldots,N, i \not=j} \bigg( \frac{1}{x_i-x_j}+\frac{1}{x_i+x_j} \bigg) \frac{\partial}{\partial x_i} f(x) \notag\\ 
&+  \sum\limits_{i=1}^N \frac{1}{x_i} \frac{\partial}{\partial x_i} f(x) + \frac{\langle \nabla_x J^B_k(x,w),\nabla f(x) \rangle}{k_1\cdot J^B_k(x,w)}
\end{align}
for all invariant $f\in C^{(2)}(\mathbb R^N)$. It now follows from
\eqref{limit-BN-functions-intro2} that for $k_1 \to \infty$,
\begin{equation}
  \widetilde L_{(k_1,k_2)}^{B,k_1 w}f(x) \to \widetilde{L}_{\infty}^{B,w}f(x) \coloneqq \sum\limits_{i=1}^N \frac{\sqrt{x_i^2w_i^2+1}}{x_i} \frac{\partial}{\partial x_i} f(x).
\end{equation}
As the $\widetilde  L_{(k_1,k_2)}^{B,k_1 w}$ are the generators of the renormalized Bessel processes
$$(\widetilde{X}^{B,k_1 w}_{t,(k_1,k_2),x} :=  X^{B,k_1 w}_{t/k_1,(k_1,k_2),x})_{t \geq 0},$$
and as these processes satisfy
stochastic differential equations associated with the  $\widetilde  L_{(k_1,k_2)}^{B,k_1 w}$, we conclude on an informal level that for $k_1 \to \infty$, the random variables 
$X^{B,k_1w}_{t/k_1,(k_1,k_2),x}$ converge in probability to the solution of the initial value problem
\begin{equation} \label{AWP y-intro}
y_i'(t)=\frac{\sqrt{y_i(t)^2w_i^2+1}}{y_i(t)}, \; y_i(0)=x_i \quad\quad(i=1,\ldots,N),\end{equation}
which can be solved easily; see Lemma  \ref{DGL y}. Hence, the following weak law is not surprising:

\begin{theorem}\label{limit-processes-intro}
Let $N \in \N$, $k_2 \geq 0$, $t \geq 0$, and $x,w \in C_N^B$. 
Then for $k_1 \to \infty$,  the random variables $X^{B,k_1 w}_{t/k_1,(k_1,k_2),x}$ tend in probability to
\begin{align*}
\left( \sqrt{x_1^2 + w_1^2 t^2 + 2t \sqrt{x_1^2 w_1^2 + 1}}, \ldots, \sqrt{x_N^2 + w_N^2 t^2 + 2t \sqrt{x_N^2 w_N^2 + 1}} \right) .
\end{align*}
\end{theorem}

Please notice that the argument above is informal only, as \eqref{limit-BN-functions-intro2}
is proved for $N=1,2$ or $k_2=0,1/2,1,2$ only, and as the arguments above are  informal. In fact, we shall prove
Theorem \ref{limit-processes-intro} in Section 5 by inserting  \eqref{limit-BN-functions-intro1} (which holds for all $k_2\ge0$)
into the densities  \eqref{density-transition-bessel-drift-gen} of the random variables $X^{B,k_1 w}_{t/k_1,(k_1,k_2),x}$
where we then use the Laplace method similar to the proof of   \eqref{limit-BN-functions-intro1} in Section 2.
When doing so, we must find the unique maximum of some concrete function, which seems to be a serious computational problem. However, having 
an idea of the limit in Theorem \ref{limit-processes-intro} by the informal argument above, we obtain a candidate for our maximation problem such that
this problem can be solved in Lemma 5.2 below.  This then will complete the proof of Theorem \ref{limit-processes-intro}.

We point out two further facts about  Theorem \ref{limit-processes-intro}.
First of all, as $J_k(cx,y)=J_k(x,cy)$ for all $c\in\mathbb C$, $x,y\in\mathbb C^N$ and all Bessel functions,
\eqref{density-transition-bessel-drift-gen} yields that Bessel processes with drift have the same 
homogeneity properties as classical Brownian motions with constant drift, i.e., the random variables $X^{B,k_1 w}_{t/k_1,(k_1,k_2),x}$ and
$X^{B,\sqrt{k_1}\cdot w}_{t,(k_1,k_2),\sqrt{k_1}x}/\sqrt{k_1}$ have the same distributions. Therefore, by Theorem \ref{limit-processes-intro}:

\begin{corollary} \label{limit-processes-intro-cor} 
Let  $N \in \N$, $k_2\ge0$, $t \geq 0$, and $x,w \in C_N^B$. 
Then for $k_1 \to \infty$,  the random variables $X^{B,\sqrt{k_1}\cdot w}_{t,(k_1,k_2),\sqrt{k_1}x}/\sqrt{k_1}$ tend in probability to
\begin{align*}
\left( \sqrt{x_1^2 + w_1^2 t^2 + 2t \sqrt{x_1^2 w_1^2 + 1}}, \ldots, \sqrt{x_N^2 + w_N^2 t^2 + 2t \sqrt{x_N^2 w_N^2 + 1}} \right) .
\end{align*}
\end{corollary}

For $x=w=0$, this corollary fits to Eq. (20) in Theorem 2 of \cite{AKM2}.

 Furthermore, we
 show in Theorem \ref{schwacher Limes K-Schlange} that the convergence in Theorem \ref{limit-processes-intro}
 and  Corollary \ref{limit-processes-intro-cor} holds locally uniformly for
$x,w\in C_N^B$ and $t>0$. We also expect that the informal approach above for $N=2$ or $k=0,1/2,1,2$
can be made precise with much better types of convergence like for corresponding results with drift $0$ in \cite{AnV1, VW}.
We also mention that Theorem \ref{limit-processes-intro} is related with freezing limits for Bessel processes without drift in
\cite{AKM1, AKM2, AM, AHV,  AnV2, GK, V1, VW}.

We now illustrate Theorem \ref{limit-processes-intro}
and  Corollary \ref{limit-processes-intro-cor} for  $N=1$ for geometric parameters where the Bessel processes with drift are the $L^2$-norms
of Brownian motions on $\mathbb R^M$ with drift:

\begin{example}\label{1-dim-example} For $N=1$, the Bessel functions of type $B$ do not depend on $k_2$, and we simply write $J_{k_1}^B$ for them.
  These functions are related to the one-dimensional normalized Bessel functions
  $j_\alpha(z)=\>_0F_1(\alpha+1; -z^2/4)$ for $\alpha\ge-1/2$, $z\in\mathbb C$ by
  $$J_{k_1}^B(z,w)=j_{k_1-1/2}(izw) \quad\quad(k_1\ge0,\> z,w\in\mathbb C);$$
  see for instance \cite{R4}. For $k_1>0$,  this Bessel function has the integral representation
\begin{align}\label{Integralformel 1D Besselfunktion}
J_{k_1}^{B_1}(x,y)=\frac{\Gamma(k_1+1/2)}{\sqrt\pi\Gamma(k_1)} \int_{-1}^{1} e^{xyt} (1-t^2)^{k_1-1}\> dt.
\end{align}
      For integers $M\in\mathbb N$ and $k_1=(M-1)/2$,
  these Bessel functions and the associated Bessel processes (without and with drift) are related with the exponential function and Brownian motions
  (without and with drift) on $\mathbb R^M$ as follows (see Subsection 3.3 of \cite{V2} for $N=1$):

  If $\sigma_M$ is the uniform distribution on the sphere $S^{M-1}\subset \mathbb R^M$, then we have
\begin{equation}\label{int-rep-N1}
  \int_{S^{M-1}} e^{\langle x,y\rangle} \> d\sigma_M(y)= J^B_{k_1}(\|x\|_2,1) \quad\quad(k_1=(M-1)/2, \> x\in  \mathbb R^M).\end{equation}
  This leads to the following connection between Brownian motions on $\mathbb R^M$ and Bessel processes with $N=1$ and with drift:
  Let $\lambda\in\mathbb R^M$ be any drift vector and $x\ge0$. Consider the uniform distribution $\sigma_{M,x}$ on the sphere $x\cdot S^{M-1}$ (for $x=0$ it
  degenerates into $\delta_0$) and modify it via $\lambda$ into
  $$d\sigma_{M,x,\lambda}(y):=\frac{1}{J^B_{k_1}(x,\|\lambda\|_2)} e^{\langle \lambda,y\rangle} \> d\sigma_{M,x}(y) \quad\quad(y\in  \mathbb R^M).$$
  Then by \eqref{int-rep-N1}, $\sigma_{M,x,\lambda}$ is a probability measure.
  If $(B^M_{t,x,\lambda})_{t\ge0}$ is a Brownian motion on  $\mathbb R^M$ with initial distribution  $\sigma_{M,x,\lambda}$ and drift $\lambda$, then by \cite{V2},
  $(\|B^M_{t,x,\lambda}\|_2)_{t\ge0}$ is a Bessel process with $k_1=(M-1)/2$ with drift $\|\lambda\|_2$ and start in $x\ge0$.
In this way,  Theorem \ref{limit-processes-intro} implies:
 \end{example}

\begin{corollary} \label{limit-processes-intro-cor-1dim} 
  Let   $t,x,w \geq 0$. For each $M\in\mathbb N$, we choose vectors $\lambda_M\in \R^M$ with $\|\lambda_M\|_2=w$ as well as 
  Brownian motions  $(B^M_{t,x,(M-1)\lambda_M/2})_{t\ge0}$ on  $\mathbb R^M$ with initial distributions $\sigma_{M,x,(M-1)\lambda_M/2}$ and drift
  vectors $(M-1)\lambda_M/2$.
Then for $M \to \infty$, the random variables $\| B^M_{2t/(M-1),x,(M-1)\lambda_M/2}  \|_2$ tend in probability to
$$ \sqrt{x^2 + w^2 t^2 + 2t \sqrt{x^2 w^2 + 1}}.$$
  \end{corollary}

The preceding example can be extended to $M\times N$-matrices over the  fields $\mathbb F=\mathbb R, \mathbb C$ or
the quaternions; see the end of Section 5. We shall also discuss the preceding results for  Dunkl kernels and Dunkl processes for $N=1$ in Section 6.

\section{Proof of Theorem \ref {limit-BN-functions-intro} for  $k_2=0,1/2,1,2$ and for $N=2$ via  integral representations}

In this section we first prove  Eqs. \eqref{limit-BN-functions-intro1} and 
\eqref{limit-BN-functions-intro2}, for $k_2=1/2,1,2$ by using the Laplace integral representation of
$J_{(k_1,k_2)}^B$ for $k_1$ sufficiently large in \cite{R4}.
For this  we first recapitulate some notations.
Let $\mathbb F=\mathbb R, \mathbb C$, or the quaternions $\mathbb H$ with real dimension $d=1,2,4$ respectively. Put $k_2:=d/2=1/2,1,2$.
Let the unitary group $U_N:=U(N,\mathbb F)$ be equipped with the normalized Haar measure $du$, and the open subset
$$D_N:=\{ v\in \mathbb F^{N\times N}: \>\> I_N-v^*v \>\>\text{positive definite}\}\subset \mathbb F^{N\times N}$$
with the usual Lebesgue measure restricted to $D_N$. Moreover, let $tr\> x$ and $\Delta(x)$
be the trace and determinant of a $N\times N$-matrix as in \cite{R4}. Put
$$\rho:=d(N-1/2)+1 \quad\text{and}\quad \mu(k_1):=k_1 + \frac{d}{2}(N-1)+1/2.$$
Assume that $k_1> (dN-1)/2$, i.e., $\mu(k_1)-\rho>-1$. Then by Corollary 4.6 and Eqs.~(4.4) and (3.12) of \cite{R4}, for all
$x,y\in C_N^B$,
\begin{equation}\label{int-rep-1}
J_{(k_1,k_2)}^B(x,y)=\frac{1}{n(k_1)}\int_{U_N}\int_{D_N} e^{Re\> tr(v^*\cdot diag(y)\cdot u^*\cdot  diag(x))}
  \Delta(I_N-v^*v)^{\mu(k_1)-\rho}\> dv \> du
\end{equation}
with some normalization   $n(k_1)>0$ which satisfies
\begin{equation}\label{kappa-sim} n(k_1)\sim \Bigl(\frac{\pi}{\mu(k_1)}\Bigr)^{dN^2/2}\sim
  \Bigl(\frac{\pi}{k_1}\Bigr)^{dN^2/2} \quad\quad \text{for}\quad k_1\to\infty
\end{equation}by Eq. (3.9) in \cite{R4}.
We now use the substitution $v^*=u_1pu_2$ in \eqref{int-rep-1}   with $u_1,u_2\in U_N$
and some diagonal matrix $p=diag(p_1,\ldots, p_N)$ with $(p_1,\ldots, p_N)\in C_N^B$. Then
$$\Delta(I_N-v^*v)=\prod_{j=1}^N (1-p_j^2) \quad\quad\text{and}\quad\quad dv=c_{d,N} \cdot du_1 \> du_2 \cdot 
\prod_{j=1}^N p_j^{h(d,N)} \> dp_1 \> \cdots\> dp_N$$
with known constants $c_{d,N}, h(d,N)>0$, and $du_1, du_2$  the  normalized Haar measure of $U_N$, and
$ dp_1 \> \cdots\> dp_N$  the Lebesgue measure. We thus can write 
the Laplace integral representation \eqref{int-rep-1} as
\begin{align}\label{int-rep-2}
  J_{(k_1,k_2)}^B(x,k_1 y)=\frac{c_{d,N}}{n(k_1)} &
  \int_{U_N}\int_{U_N}\int_{U_N} \int_{C_N^B\cap [0,1]^N} H(x,y,p, u_1,u_2,u_3)^{k_1} \cdot \\
&\quad\quad \prod_{j=1}^N p_j^{h(d,N)} \cdot\prod_{j=1}^N (1-p_j^2)^{-(dN+1)/2} \> dp \> du_1 \> du_2\> du_3
\notag  \end{align}
with
\begin{equation}\label{int-rep-3}
  H(x,y,p, u_1,u_2,u_3):= e^{Re\> tr(u_1 \cdot diag(p) \cdot  u_2 \cdot  diag(x) \cdot u_3  \cdot diag(y))}\cdot\prod_{j=1}^N (1-p_j^2).
  \end{equation}
In order to compute the limits \eqref{limit-BN-functions-intro1} and \eqref{limit-BN-functions-intro2} we now study the maxima of 
$H$ w.r.t. $u_1,u_2,u_3\in U_N$ and $p\in[0,1]^N$ for given $x,y\in C_N^B$.
For this we first recapitulate the following estimates for the ordered singular values $0\le s_{j+1}(A)\le s_j(A)$ ($j=1,\ldots, N-1$)
of matrices $ A\in \mathbb F^{N\times N}$ from \cite{HJ}; see  Theorems 3.3.13(a') and 3.3.14(a) there:

\begin{lemma}\label{estimates-singular-values} Let $ A,B\in \mathbb F^{N\times N}$. Then,
  $$ Re\> tr(AB)\le \sum_{j=1}^N s_j(AB) \quad\text{and}\quad
    \sum_{j=1}^k s_j(AB) \le \sum_{j=1}^k s_j(A)s_j(B) \quad\text{for}\quad k=1,\ldots,N .$$
\end{lemma}

We now need the following extension of von Neumann's trace inequality:

\begin{lemma}\label{maxima-cone1}
  For all $x,y,p\in C_N^B$ and $u_1,u_2,u_3\in U_N$,
  \begin{equation}\label{inequality-trace}  Re\> tr(u_1 \cdot diag(y) \cdot  u_2 \cdot  diag(p) \cdot u_3  \cdot diag(x)) \le
    tr( diag(y) \cdot  diag(p) \cdot  diag(x))=\sum_{j=1}^Np_jx_jy_j.\end{equation}
  Moreover, if $x_1>\ldots>x_N>0$ and $p_1>\ldots>p_N$, and if equality holds in \eqref{inequality-trace}, then
  \begin{equation}\label{equality-trace}
     Re\> (u_1 \cdot diag(y) \cdot  u_2 \cdot  diag(p) \cdot u_3 )_{j,j}=p_jy_j \quad\text{for}\quad j=1,\ldots,N.
    \end{equation}
\end{lemma}

\begin{proof} Let $A:=u_1 \cdot diag(y)  \cdot u_2 \cdot  diag(p) \cdot u_3$. 
  Then by Lemma \ref{estimates-singular-values},
  \begin{equation}\label{inequality-trace1}  Re\> tr(A\cdot diag(x))\le \sum_{j=1}^N s_j(A\cdot diag(x))\le \sum_{j=1}^N s_j(A)x_j.
    \end{equation}
  As the singular values are not changed after multiplication with matrices in $U_N$ from the left- or right-hand side, we have
  $s_j(A)=s_j(diag(y)  \cdot u_2 \cdot  diag(p))$.
   As by the same reason, $s_j(diag(y) \cdot u_2)=y_j $ for all $j$, we conclude from Lemma \ref{estimates-singular-values}
   that for $k=1,\ldots,N$,
   \begin{equation}\label{inequality-trace2} \sum_{j=1}^k s_j(A)\le \sum_{j=1}^kp_jy_j.\end{equation}
   As with $x_{N+1}(A):=0$,
$$\sum_{j=1}^Nx_jy_jp_j= \sum_{k=1}^N (x_k-x_{k+1})\sum_{j=1}^k p_jy_j  \quad\text{and}\quad
   \sum_{j=1}^Nx_js_j(A)= \sum_{k=1}^N (x_k-x_{k+1})\sum_{j=1}^k s_j(A),$$
     we conclude from \eqref{inequality-trace2} and the monotonicity of the $x_k$ that
 \begin{equation}\label{inequality-trace3}   \sum_{j=1}^N s_j(A)x_j\le\sum_{j=1}^N x_jy_jp_j.
 \end{equation}
 This and \eqref{inequality-trace1} now lead to \eqref{inequality-trace}.

 Assume now  $x_1>\ldots>x_N>0$ and $p_1>\ldots>x_N$, and that equality holds in \eqref{inequality-trace}. The arguments above
 then imply that for all $k=1,\ldots,N$,
$$\sum_{j=1}^k p_jy_j=\sum_{j=1}^k s_j(A) \quad\text{and thus}\quad  p_ky_k=s_k(A)=s_k(diag(y)  \cdot u_2 \cdot  diag(p)).$$
 If we consider the 2-norms of the images of the first unit vector $e_1\in\mathbb R^N$, we obtain that $e_1$
 is mapped by $u_2$ into the eigenspace of the eigenvalue $y_1$ of $diag(y)$. As all eigenspaces are one-dimensional, we see
 inductively, that for all $k$, $e_k$ is mapped by $u_2$ into the eigenspace of the eigenvalue $y_k$ of $diag(y)$, i.e.,
  $u_2$ is  diagonal  which commutes with $diag(y)$. This means that we may assume  $u_2=I_N$
  and  $A=(a_{i,j})_{i,j}=u_1 \cdot diag(y)   diag(p) \cdot u_3$.
 If we have equality in \eqref{inequality-trace}, we thus obtain 
 \begin{align} \label{inequality-trace4} 
   \sum_{k=1}^N (x_k-x_{k+1})\sum_{j=1}^k a_{j,j}&= \sum_{j=1}^N x_j a_{j,j} = Re\> tr(A\cdot diag(x))= \sum_{j=1}^Nx_jy_jp_j \notag\\
   &=
 \sum_{k=1}^N (x_k-x_{k+1})\sum_{j=1}^k p_jy_j .\end{align} 
For  $k=1,\ldots,N$ we now consider the matrices $P_k:=diag(1,\ldots,1,0,\ldots,0)$ in which $1$ appears $k$-times.
Then by Lemma \ref{estimates-singular-values},
\begin{align} \sum_{j=1}^k a_{j,j} &= Re\> tr(A\cdot P_k)=  Re\> tr(u_1 \cdot diag(y)  \cdot  diag(p) \cdot u_3 \cdot P_k)\notag\\
  &=  Re\> tr( u_3 \cdot P_k\cdot u_1 \cdot diag(y) diag(p))  \le
  \sum_{j=1}^N y_jp_j s_j( u_3 \cdot P_k\cdot u_1) = \sum_{j=1}^k y_jp_j  .\notag\end{align} 
  We thus conclude from \eqref{inequality-trace4} and the strict monotonicity of the $x_k$ that
  for all $k$, $\sum_{j=1}^k p_jy_j= \sum_{j=1}^k a_{j,j}$ and thus  $ a_{j,j}=p_jy_j$ for all $j$ as claimed.

   We mention that this proof also works for $\mathbb F=\mathbb H$, as here also
  $ Re\> tr(AB)= Re\> tr(BA)$ holds.
\end{proof}

\begin{lemma}\label{maxima-cone2}
  For $x,y\in C_N^B$, the function $H$ in \eqref{int-rep-3} has a maximum at $u_1=u_2=u_3=I_N$ and $p=(p_1,\ldots,p_N)$ with
  $$p_j:=\frac{1}{x_jy_j}(\sqrt{x_j^2y_j^2+1}\> -1)$$ for $j=1,\ldots,N$. Moreover,
  $$\max_{u_1,u_2,u_3\in U_N, \> p\in[0,1]^N} H(x,y,p, u_1,u_2,u_3)= \prod_{j=1}^N \Bigl( e^{\sqrt{ x_j^2y_j^2+1}\> -1}\frac{2}{x_j^2y_j^2}(\sqrt{x_j^2y_j^2+1}\> -1)\Bigr) $$
\end{lemma}

\begin{proof} Clearly, for $x\ge0$, the function $f(t):=  e^{x t}(1-t^2)$  is nonnegative on $[-1,1]$ with a unique maximum at
  $$t(x):=\frac{1}{x}(\sqrt{x^2+1}\> -1) \quad\quad (\text{with}\quad t(0)=0) $$ with 
  $$f(t(x))=e^{\sqrt{ x^2+1}\> -1}\frac{2}{x^2}(\sqrt{x^2+1}\> -1) \quad\quad (\text{with}\quad f(t(0))=1).$$
This and Lemma \ref{maxima-cone1} now yield the claim.
  \end{proof}

We also need the following variant of the Laplace method which is likely to be known:

\begin{lemma}\label{Laplace-method-roots}
Let  $X,Y$ be compact metric spaces, $\mu$ a positive finite Borel measure on $Y$ with $supp\>\mu=Y$,
$g,h:X \times Y \to [0,\infty[$  continuous functions, and
$$f_k(x):= \Bigg( \int\limits_Y (g(x,y))^k \cdot h(x,y) d\mu(y) \Bigg)^{1/k} \quad (x\in X,\>\>k \ge 1 ).$$
 Assume that for
 all $x \in X$, $V(x) \coloneqq \{y \in Y \, | \, h(x,y)>0\}\ne\emptyset$.
Moreover, let $f_{\infty}:X \to ]0,\infty[\,$, $f_{\infty}(x) \coloneqq \max\limits_{y \in \overline{V(x)}} g(x,y)$ be continuous.
    Then $f_k$ tends uniformly on $X$ to $f_{\infty}$ for $k\to\infty$.
\end{lemma}

\begin{proof}
  We first notice that for $k\to\infty$ and  uniformly for $x\in X$,
  \begin{equation}\label{est-upper}
f_k(x)\le f_{\infty}(x)\cdot \Bigl( \int_Y  h(x,y) d\mu(y) \Bigr)^{1/k} \to f_{\infty}(x).
  \end{equation}

  For a lower estimate we observe that it suffices to show that for each  $x \in X$ and $\epsilon>0$ there is a neighborhood $U_x$
  of $x$ with $f_k(\tilde x)\ge  f_{\infty}(\tilde x)-\epsilon$ for $\tilde x\in U_x$ and $k$ sufficiently large.
  In fact, if this holds, and as finitely many $U_x$ form a covering of the compactum $X$, we find some $k_0$ such that for all $k\ge k_0$ we have
  $f_k(\tilde x)\ge f_{\infty}(\tilde x)-\epsilon$ for all $\tilde x\in X$. This and \eqref{est-upper} then yield the claim.

  We now fix $x\in X$ and let $\tilde\epsilon>0$. Choose some neighborhood $U_x\times U_y\subset X\times Y$ of $(x,y(x))$ with some 
  $y(x)\in Y$ with $g(x,y(x))=f_\infty(x)$ such that for all $(\tilde x,\tilde y)\in U_x\times U_y$ we have
  $g(\tilde x,\tilde y)\ge f_\infty(\tilde x) -\tilde\epsilon$
  and $V(\tilde x)\cap U_y\ne\emptyset$.
  We thus obtain for $\tilde x\in U_x$ that
  $$f_k(\tilde x)\ge \Bigl( \int_{U_y} g(\tilde x, \tilde y)^k   h(\tilde x,\tilde y) d\mu(\tilde y) \Bigr)^{1/k} \ge
  (f_\infty(\tilde x) -\tilde\epsilon)\cdot \Bigl( \int_{U_y} h(\tilde x,\tilde y) d\mu(\tilde y) \Bigr)^{1/k},$$
where by our assumptions on $U_x\times U_y$, we have $\int_{U_y} h(\tilde x,\tilde y) d\mu(\tilde y)>0$ and thus 
$$ \Bigl(\int_{U_y} h(\tilde x,\tilde y) d\mu(\tilde y)\Bigr)^{1/k}\to 1 \quad\text{for} \quad k\to\infty.$$
These facts for a suitable
$\tilde\epsilon>0$ 
immediately lead to the neighborhood $U_x$ for $x\in X$ and $\epsilon$ as claimed above. This completes the proof.
\end{proof}

If we now combine the integral representation \eqref{int-rep-2}, Lemma  \ref{maxima-cone2}, and \ref{Laplace-method-roots} we obtain
Eq.~\eqref{limit-BN-functions-intro1} in Theorem 
\ref{limit-BN-functions-intro} for $k_2=1/2,1,2$.  In a similar way,  Eq.~\eqref{limit-BN-functions-intro2} in Theorem 
\ref{limit-BN-functions-intro} is, for these $k_2$, a consequence of Lemma \ref{maxima-cone2}
and the following variant of  Lemma \ref{Laplace-method-roots}:

\begin{lemma}\label{Laplace-method-fraction}
  Let $X,Y$ be compact metric spaces, $\mu$ a positive finite Borel measure on $Y$ with $supp\>\mu=Y$,
  and $f,g,h:X\times Y\to[0,\infty[$
      continuous functions. For $x\in X$ let
$$M_x:=\{y\in Y: g(x,y)=m(x):=\max_{z\in Y} g(x,z)\}.$$
Assume that for all $x\in X$,
\begin{enumerate}\itemsep=-1pt
\item[\rm{(1)}] the functions
  $y\mapsto f(x,y)$ and $y\mapsto h(x,y)$ are constant and positive on $M_x$, and that
\item[\rm{(2)}] for $y\in M_x$,  $g(x,y)f(x,y)h(x,y)>0$.
  \end{enumerate}
Then, uniformly for $x\in X$ and $y(x)\in M_x$, 
$$\frac{\int_Y g(x,y)^k h(x,y)\> d\mu(y)}{\int_Y g(x,y)^k f(x,y)\> d\mu(y)} \to  \frac{h(x,y(x))}{f(x,y(x))}
\quad\quad\text{for}\quad k\to\infty. $$
 \end{lemma}

\begin{proof} The function $m:X\to[0,\infty[$ is continuous by a standard argument and positive by condition (2).
      As we can replace  $g$ by the function $g(x,y)/m(x)$ in the lemma, we may assume
      $m= 1$ on $X$.
      
Moreover, the functions  $F(x):= f(x,y(x))$ and $H(x):= h(x,y(x))$
with $y(x)\in M_x$   are well-defined by (1) and also continuous on $X$. In fact, 
  for any sequence $(x_n)_n\in X$ with limit $x\in X$, we 
  choose a sequence $(y_n)_n\in Y$ with $y_n\in M_{x_n}$ for $n\in\mathbb N$.
  This sequence has some subsequence which converges to
some $y\in Y$. As $g(x_n,y_n)=1$, we obtain $g(x,y)=1$, i.e., $y\in M_x$. Hence, for some subsequence,
$F(x_{n_l})=f(x_{n_l},y_{n_l})\to f(x,y)=F(x)$ for $l\to\infty$. As we may apply this argument to any subsequence
$(x_n)_n$ we conclude that $F(x_n)\to F(x)$ for $n\to\infty$, i.e., $F$ and, by the same reason, $H$ are continuous.
$F$ and $H$  are also positive by condition (2).

The sequence argument above  also ensures that  $S:=\{(x,y)\in X\times Y: \> y\in M_x\}$ is compact.
With the function  $Q:=H/F$  we have $h(x,y)= Q(x)f(x,y)$ on  $S$.
Now take $\epsilon>0$. Using uniform continuity, we find
an open set $U\in  X\times Y$ with $S\subset U$ such that for $(x,y)\in U$, $f(x,y)>0$ and
 $$|h(x,y)-Q(x)f(x,y)|\le \epsilon f(x,y).$$
Choose $\rho<1$ with $g(x,y)\le \rho$ for $(x,y)\in (X\times Y)\setminus U$.
Then, for $x\in X$ and $W_x:=\{y\in Y:\> (x,y)\in U\}$,
\begin{align}\label{estimate-main-fraction}
\Bigl|\int_Y g(x,y)^k h(x,y)d\mu(y)&- Q(x)\int_Y g(x,y)^k f(x,y)\> d\mu(y)\Bigr| \\
  &=
\Biggl|\Biggl(\int_{W_x}+ \int_{Y\setminus W_x}\Biggr) g(x,y)^k (h(x,y)- Q(x)f(x,y))\> d\mu(y)\Biggl|\notag\\
&\le   \epsilon \int_Y g(x,y)^k f(x,y) \> d\mu(y) + C_1 \rho^k
\notag\end{align}
with some constant  $C_1>0$. We next take some constant $R\in]\rho,1[$. For $x\in X$ we now consider neighborhoods
      $U_x\times V_x\subset X\times Y$ of
      $(x,y(x))$ with $y\in M_x$  such that for all $(\tilde x,\tilde y)\in U_x\times V_x$ we have $g(x,y)>R$.
      Then there are finitely many $x_1,\ldots,x_n\in X$ with $X=U_{x_1}\cup\ldots\cup U_{x_n}$.
      Hence, for each $x\in X$ there exists some $j$ with $V_{x_j}\subset\{y\in Y:\> g(x,y)>R\}$. Hence, for each $x\in X$,
      $$\int_Y g(x,y)^k f(x,y) \> d\mu(y)\ge R^k\int_{V_j} f(x,y) \> d\mu(y)\ge C_2 R^k$$
      with some constant $C_2>0$.  \eqref{estimate-main-fraction} now implies 
      that for all $x\in X$, $y\in M_x$, and $k$,
      $$\Biggl|\frac{\int_Y g(x,y)^k h(x,y)\> d\mu(y)}{\int_Y g(x,y)^k f(x,y)\> d\mu(y)} - \frac{h(x,y(x))}{f(x,y(x))}\Biggr|\le \epsilon + \frac{C_1}{C_2} \Bigl(\frac{\rho}{R}\Bigr)^k.$$
      This implies the claim.
      \end{proof}

\begin{proof}[Proof of  Eq.~\eqref{limit-BN-functions-intro2} in Theorem 
    \ref{limit-BN-functions-intro} for $k_2=1/2,1,2$]
  
 Let $x,y$ be in the interior of $C_N^B$ and $j=1,\ldots,N$.
 By the  integral representation \eqref{int-rep-2}
we have
\begin{align}
 & \frac{\frac{d}{dx_j} J_{(k_1,k_2)}^B(k_1 x,y) }{k_1\cdot J_{(k_1,k_2)}^B(k_1 x,y)}\notag\\
  &= \frac{ \int_{U_N^3}\int_{C_N^B\cap [0,1]^N} H(x,y,p, u_1,u_2,u_3)^{k_1} f(p) \phi(y,p,u_1,u_2,u_3) \> dp \> d(u_1,u_2,u_3)}{
\int_{U_N^3}\int_{C_N^B\cap [0,1]^N} H(x,y,p, u_1,u_2,u_3)^{k_1} f(p) \> dp \> d(u_1,u_2,u_3)}\notag\end{align}
with $f(p):=\prod_{l=1}^N \Bigl( p_l^{h(d,N)} (1-p_l^2)^{-(dN+1)/2}\Bigr)$ and
$$  \phi(y,p,u_1,u_2,u_3):=  Re\> (u_3  \cdot diag(y)\cdot u_1 \cdot diag(p) \cdot  u_2 )_{j,j}.$$
 Then, the unique optimal $p$ in the function $H$ from Lemma \ref{maxima-cone2}
satisfies $p_l:=\frac{1}{x_ly_l}(\sqrt{x_l^2y_l^2+1}\> -1)$ for $l=1,\ldots,N$
where by the conditions on $x,y$
we have
$p_1>\ldots>p_N>0$. We conclude from the second statement in Lemma \ref{maxima-cone1}
that  $\phi(y,p,u_1,u_2,u_3)$ is independent from the optimal unitary matrices $u_1,u_2,u_3$.
We thus can apply Lemma \ref{Laplace-method-fraction} to the limit $k_1\to\infty$ (where we possibly have to shift some suitable power of 
$\prod_{l=1}^N (1-p_l^2)$ from $H$ to $f$). This and
$$\phi(y,p,u_1,u_2,u_3)= \frac{\sqrt{x_j^2y_j^2+1}-1}{x_j} \quad\text{for the optimal}\quad p,u_1,u_2,u_3$$
complete the proof of  Eq.~\eqref{limit-BN-functions-intro2}.
\end{proof}

\begin{proof}[Proof of Theorem \ref{limit-BN-functions-intro} for $k_2=0$]
 
Let $k_2=0$ and $N\ge2$. In this case,
the definitions of the associated Dunkl operators, Dunkl kernels, and Bessel functions (see e.g. \cite{D,R3}) imply that
for  $k_1\ge0$, $x=(x_1,\ldots,x_N),y=(y_1,\ldots,y_N)\in C_N^B$, the one-dimensional  Dunkl kernels $E_{k_1}$ (see also Section 6 below),
and the symmetric group $S_N$, we have
\begin{equation}\label{Bessel-fct-k20}
  J^{B_N}_{(k_1,0)}(x,y)= \frac{1}{N! \cdot 2^N}\sum_{\sigma\in S_N}\sum_{p_1,\ldots,p_N=\pm 1}\prod_{j=1}^N E_{k_1}(x_j, p_j\cdot y_{\sigma(j)})
  =\frac{1}{N!}\sum_{\sigma\in S_N} \prod_{j=1}^N  J^{B_1}_{k_1}(x_j,  y_{\sigma(j)}).
\end{equation}
We next check that $\partial_x \partial_y ln J^{B_1}_{k_1}(x,y)>0$ for $k_1,x,y > 0$. For this we use
the integral representation \eqref{Integralformel 1D Besselfunktion} for $J_{k_1}^{B_1}$ and
observe by the quotient rule for derivatives that
\begin{align*}
  \partial_x \partial_y ln J^{B_1}_{k_1}(x,y)
  = \frac{\int_{-1}^1 (1+xyt)te^{xyt}(1-t^2)^{k_1-1}\> dt}{\int_{-1}^1 e^{xyt}(1-t^2)^{k_1-1}\> dt} -
  xy \cdot \Bigg( \frac{ \int_{-1}^1 te^{xyt}(1-t^2)^{k_1-1}\> dt}{\int_{-1}^1 e^{xyt}(1-t^2)^{k_1-1}\> dt}\Bigg)^2.
\end{align*}
If $X$ is a $[-1,1]$-valued random variable with distribution $c_{k_1,x,y}e^{xyt}(1-t^2)dt$
where $c_{k_1,x,y}=(\int_{-1}^1 e^{xyt}(1-t^2)^{k_1-1}\> dt)^{-1}>0$, we see as claimed that 
$$\partial_x \partial_y ln J^{B_1}_{k_1}(x,y)=xy \cdot Var(X)+E(X) > E(X)=2c_{k_1,x,y} \cdot \int_0^1 t \cdot \sinh(xyt) \cdot (1-t^2) \> dt >0.$$

We thus obtain that for $k_1>0$ and  $x,y$ in the interior of $C_2^B$,
$$\ln J^{B_1}_{k_1}(x_1,y_1)+\ln J^{B_1}_{k_1}(x_2,y_2)-\ln J^{B_1}_{k_1}(x_1,y_2)-\ln J^{B_1}_{k_1}(x_2,y_1)
= \int\limits_{x_2}^{x_1} \int\limits_{y_2}^{y_1} \partial_x \partial_y ln J^{B_1}_{k_1}(x,y)\>dy \> dx>0.$$
This shows that for $x,y \in C_N^B$ and $k_1 > 0$, the matrix $C:=(-\ln J^{B_1}_{k_1}(x_i,y_j))_{i,j=1,\ldots,N}$ is a Monge matrix
in the sense of Definition 5.5 in \cite{BDM},
i.e. $C_{i,j} + C_{k,l} \leq C_{i,l} + C_{j,k}$ for all integers $1 \leq i < k \leq N$ and $1 \leq j < l \leq N$,
where this inequality is even strict  for $x,y$ in the interior of $C_N^B$.
Proposition 5.7 in \cite{BDM} now shows  for $x,y \in C_N^B$ that in the sum in the end of \eqref{Bessel-fct-k20}, the summand
$\prod_{j=1}^N  J^{B_1}_{k_1}(x_j,  y_{\sigma(j)})$ is maximal for the identity $\sigma\in S_N$
where the proof  of this proposition shows that this maximum is unique for $x,y$ in the interior of $C_N^B$.
This, Lemma \ref{Laplace-method-roots}, and
  \eqref{limit-BN-functions-intro1} for $N=1$ (which is shown just above)
  now readily lead to \eqref{limit-BN-functions-intro1} for $N\ge2$ and $k_2=0$.
In the same way, Lemma  \ref{Laplace-method-fraction} leads to \eqref{limit-BN-functions-intro2} for $N\ge2$ and $k_2=0$.
Clearly, the limits hold locally uniformly.
\end{proof}

\begin{proof}[Proof of Theorem \ref{limit-BN-functions-intro} for the root system $B_2$]
  
  We first recapitulate the Laplace-type integral representation for $J_{(k_1,k_2)}^{B_2}$ of Demni (see \cite{De, AD}) where there the
  parameters $k_1,k_2$ are exchanged compared to our notation. For $x,y\in \mathbb R^2$  in our notation, we then have
  \begin{equation}\label{demni-laplaceb2}
    J_{(k_1,k_2)}^{B_2}(x,y)= c_{(k_1,k_2)}\int_{[-1,1]^2} J_{k_1+k_2}^{B_1}\Bigl(\sqrt{Z_{x,y}(u,v)/2},1\Bigr)\cdot (1-u^2)^{k_2-1}(1-v^2)^{k_1-1} \> du\> dv
  \end{equation}
  with
  $$Z_{x,y}(u,v):= (x_1^2+x_2^2)(y_1^2+y_2^2) +u     (x_1^2-x_2^2)(y_1^2-y_2^2) +4v  x_1 x_2 y_1 y_2$$
  and
  $$c_{(k_1,k_2)} =\frac{\Gamma(k_1+1/2)\Gamma(k_2+1/2)}{\pi \Gamma(k_1)\Gamma(k_2)}.$$
  Clearly, this integration w.r.t.~$u,v$ degenerates for $k_1=0$ or $k_2=0$ into an integration  w.r.t.~the measure
$(\delta_1+\delta_{-1})/2$.
  We now fix $k_2\ge0$, $x,y\in C_2^B$, and consider  $k_1\to\infty$.
As $ J_{k_1+k_2}^{B_1}$ is increasing on $[0,\infty[$ (see e.g.~\eqref{Integralformel 1D Besselfunktion} below),
    $J_{k_1+k_2}^{B_1}(\sqrt{Z_{x,y}(u,v)/2})$ is increasing in $u$ and   $v$ on $[-1,1]$.
    Using \eqref{limit-BN-functions-intro1} for $N=1$, we can write the integrand in \eqref{demni-laplaceb2}, after replacing $y$ with $k_1 y$, as
    \begin{equation}\label{demni-laplaceb22} exp\bigl( k_1( H(x,y,u,v) +o(1))\bigr)
        \cdot(1-v^2)^{-1} (1-u^2)^{k_2-1}
        \end{equation}
 for $k_1\to\infty$     with
 \begin{equation}\label{demni-exp-form}
   H(x,y,u,v) :=    \Bigl( \sqrt{1+Z_{x,y}(u,v)/2}-1\Bigr) + \ln\Bigl(\frac{\sqrt{1+Z_{x,y}(u,v)/2}-1}{Z_{x,y}(u,v)/2}\Bigr) + \ln (1-v^2) + \ln(2)
 \end{equation}
 where $o(1)$ holds locally uniformly in $x,y,u,v$ for $k_1\to\infty$.
 By elementary calculus,
      $(u,v)\mapsto H(x,y,u,v)$ has a unique maximum on $[-1,1]^2$ at some point $(1,v_0)$
      with  $v_0\in [0,1[$ independent from $k_2$. 
       A slight extension of Lemma \ref{Laplace-method-roots}  (due to the $o(1)$-term
        in \eqref{demni-laplaceb22}) now shows
      that the limit in \eqref{limit-BN-functions-intro1}
      exists locally uniformly in $x,y\in C_2^B$ and is independent from $k_2$. Hence, \eqref{limit-BN-functions-intro1} for general $k_2>0$
      follows from the known special cases $k_2=0,1/2,1,2$. Moreover, these arguments also yield that the first limit in
      \eqref{limit-BN-functions-intro3} exists for $\nu_1,\nu_2>0$ and locally  uniformly for $x,y\in C_2^B$.

      The proof is more involved for  \eqref{limit-BN-functions-intro2}. Here \eqref{demni-laplaceb2}
      and the integral representation \eqref{Integralformel 1D Besselfunktion}
 lead to
      \begin{align}\label{demni-laplaceb23}    \partial_{x_1}J_{(k_1,k_2)}^{B_2}(x,k_1y)= &k_1 c_{(k_1,k_2)} \frac{\Gamma(k_1+k_2+1/2)}{\sqrt\pi\Gamma(k_1+k_2)} 
\int_{[-1,1]^2}\int_{[-1,1]}
        e^{t k_1\cdot\sqrt{Z_{x,y}(u,v)/2} }t(1-t^2)^{k_1+k_2-1}\\
&(1-u^2)^{k_2-1} (1-v^2)^{k_1-1} \frac{x_1((y_1^2+y_2^2) +u (y_1^2-y_2^2))+ 2v x_2 y_1 y_2}{\sqrt{Z_{x,y}(u,v)/2}} \> dt \> du\> dv\notag
 \end{align}
      and a similar, simpler expression for $J_{(k_1,k_2)}^{B_2}(x,k_1 y)$. In both integrals the terms, which appear with power $k_1$, are equal,
      and it follows easily by elementary calculus from the preceding consideration regarding \eqref{limit-BN-functions-intro1}
      that this term has a unique maximum for 
      $u,v,t\in[-1,1]$. We thus obtain from Lemma \ref{Laplace-method-fraction}  that also the limit in
\eqref{limit-BN-functions-intro2} exists for $N=2$ and all $k_2>0$ locally uniformly for $x,y$ in the interior of 
 $C_2^B$. In order to identify that limit we use that the limit 
in \eqref{limit-BN-functions-intro2} is the logarithmic derivative of the limit in \eqref{limit-BN-functions-intro1}. 
As we know the limit in \eqref{limit-BN-functions-intro1}, it follows from a standard result in
calculus that also the limit in \eqref{limit-BN-functions-intro2} is the correct one, 
i.e., \eqref{limit-BN-functions-intro2} holds for $N=2$ and all $k_2\ge0$ as claimed.
These arguments also show that for all $\nu_1,\nu_2>0$, the second limit in
\eqref{limit-BN-functions-intro3} exists locally uniformly w.r.t. $x,y$ in the interior of $C_2^B$.
\end{proof}

\begin{remark} Even for the root system $B_2$ we are not able to compute the limits \eqref{limit-BN-functions-intro3}
 for general parameters $\nu_1,\nu_2>0$, as the optimization problems above lead to complicated algebraic equations.
 For the root system $A_{N-1}$ there exists a recursive integral representation in \cite{Amri},
 which is based on a corresponding result 
 for Jack polynomials in \cite{OO}. As for the general $B_2$-case it seems to be impossible to identify the limits
 as in Theorem \ref{limit-BN-functions-intro} for $N\ge3$ (except for particular $x,y\in C_N^A$).

 On the other hand, there is a further particular case for the root system $B_2$ where the limits can be computed, namely 
 $k_2\to\infty$ with $k_1\ge0$ fixed. The reason  is  that for $B_2$ the root sets
 $\{\pm e_i:\> i=1,2\}$ and $\{\pm e_1\pm e_2\}$ change their roles under a $\pi/4$-degree rotation. This is geometrically trivial and
 can also be
  checked  via \eqref{demni-laplaceb2} and the  coordinates $(\tilde x_1,\tilde x_2)=\frac{1}{\sqrt 2}(x_1+x_2, x_1-x_2)$. As now
$k_1,k_2$ change their roles,
\eqref{limit-BN-functions-intro1} shows that for all $x,y\in C_N^{B_2}$ and  $k_1\ge0$,
\begin{align}\label{rotated limit} 
\lim_{k_2 \to \infty} \Bigg( J_{(k_1,k_2)}^{B_2}(x,k_2 y)& \Bigg)^{1/k_2}
= 4e^{\sqrt{1+(x_1+x_2)^2(y_1+y_2)^2/4}-1}\cdot e^{\sqrt{1+(x_1-x_2)^2(y_1-y_2)^2/4}-1}\\
&\cdot\frac{1}{(\sqrt{1+(x_1+x_2)^2(y_1+y_2)^2/4}+1)( \sqrt{1+(x_1-x_2)^2(y_1-y_2)^2/4}+1)}.\notag\end{align}
Clearly, this argument can be also applied to the limit \eqref{limit-BN-functions-intro2}.
\end{remark}

\section{A limit result for Jack polynomials}

In this section we present a limit result for Jack polynomials which will 
 be  central for our proof of  \eqref{limit-BN-functions-intro1} 
 for general $k_2>0$.
We first recapitulate some facts on partitions and Jack polynomials from  \cite{M, S, DES}.
 Let $\Lambda_N \coloneqq C_N^B \cap \Z^N$  the set of all partitions of length at most $N$.
 As usual, we identify partitions $\lambda \in \Lambda_N$ with $(\lambda,0,\ldots,0)\in\Lambda_{M} $ for  $M > N$.
Furthermore, $l(\lambda) \coloneqq \#\{i=1,\ldots,N \, | \, \lambda_i>0\}$ is the length of $\lambda\in \Lambda_N$, and
$|\lambda| \coloneqq \sum\limits_{i=1}^N \lambda_i$ is called its weight. 
$\lambda' \in \Lambda_{\lambda_1}$ with $\lambda'_i \coloneqq \#\{j=1,\ldots,N \, | \, \lambda_j \geq i \}$
is called the conjugate partition of $\lambda$.
Moreover, we use the notation
\begin{align}\label{def-pairs}
(i,j) \in \lambda 
:&\iff (i,j) \in \{1,\ldots,N\} \times \{1,\ldots,\lambda_1\} \text { and } j \leq \lambda_i \\
&\iff (i,j) \in \{1,\ldots,N\} \times \{1,\ldots,\lambda_1\} \text { and } i \leq \lambda_j'. \notag
\end{align}
Furthermore, we define
\begin{equation}\label{def-c-lambda}
c_{\lambda}(\alpha) \coloneqq \prod\limits_{(i,j) \in \lambda} (\alpha(\lambda_i-j)+\lambda_j'-i+1), \quad
c'_{\lambda}(\alpha) \coloneqq \prod\limits_{(i,j) \in \lambda} (\alpha(\lambda_i-j+1)+\lambda_j'-i),
\end{equation}
and  the generalized Pochhammer symbol
\begin{align*}
(a)_{\lambda}^{(\alpha)} 
\coloneqq \prod\limits_{i=1}^N \Big(a-\frac{i-1}{\alpha}\Big)_{\lambda_i}
= \prod\limits_{(i,j)\in\lambda}\Big(a-\frac{i-1}{\alpha}+j-1\Big) \quad \text{for}\quad a\in\R,\> \alpha>0,\> \lambda\in\Lambda_N.
\end{align*}
We also use the dominance order
\begin{align*}\lambda \leq \tau :\iff \bigg(\forall{i=1,\ldots,N:} \, \sum\limits_{j=1}^i \lambda_j \leq \sum\limits_{j=1}^i \tau_j \bigg) \text{ and } |\lambda|=|\tau|.
\end{align*}

We now consider the Jack polynomials $P_{\lambda}^{(\alpha)}$ and  $C_{\lambda}^{(\alpha)}$ for $\alpha>0$ and  $\lambda\in \Lambda_N$,
where these polynomials differ by constant factors only. The $P_{\lambda}^{(\alpha)}$ are homogeneous of degree $|\lambda|$ and symmetric.
Moreover,  they are characterized as eigenfunctions of the differential operators
$$H_{1/\alpha} \coloneqq \sum\limits_{i=1}^N x_i^2 \frac{d^2}{dx_i^2}+
\frac{2}{\alpha} \cdot \sum\limits_{i,j=1,\ldots N; i \ne j} \frac{x_i^2}{x_i-x_j} \frac{d}{dx_i}$$
with the eigenvalues
$E_{\lambda,1/\alpha} \coloneqq \sum\limits_{i=1}^N \lambda_i \cdot (\lambda_i-1+\frac{2}{\alpha} \cdot (N-i)).$
Moreover,
\begin{equation}\label{def-p-m-lambda}
  P_{\lambda}^{(\alpha)}= m_{\lambda}+
  \sum_{\mu \in \Lambda_N; \mu \leq \lambda, \> \mu\ne\lambda} c_{\lambda,\mu}(1/\alpha) \cdot m_{\mu}
  \end{equation}
 with suitable coefficients 
$c_{\lambda,\mu}(1/\alpha) \in \R$ and the symmetric polynomials
$m_{\lambda}(x) \coloneqq \sum_{\mu\in S_N \lambda} x^{\mu}$
with the symmetric group $S_N$. 
The renormalized Jack polynomials  are defined by
\begin{equation}\label{norming-p-c}
  C^{(\alpha)}_{\lambda}=\frac{|\lambda|!\alpha^{|\lambda|} }{ c'_{\lambda}(\alpha)} \cdot P_{\lambda}^{(\alpha)};\end{equation}
see e.g.~ Definition 2.10 and Table 6 in [DES]. These polynomials  satisfy
\begin{equation}\label{sum-identity-c}
\sum_{\mu \in \lambda_N:\>  |\mu|=n} C_{\mu}^{(\alpha)}(x)=\Bigg(\sum\limits_{i=1}^N x_i\Bigg)^n \quad(n \in \N_0).\end{equation}
We next summarize some known facts on the Jack polynomials and  symmetric polynomials:

\begin{lemma}\label{Lemma Jackpolynom}
Let $N \in \N$, $k_2>0$, and $\mu,\lambda \in \Lambda_N$ with $\mu \leq \lambda$. 
Then, for ${\bf 1}=(1,\ldots,1)\in\R^N$,  $c_{\lambda,\lambda}(k_2) \coloneqq 1$, and  $x \in C_N^B$,
\begin{align}
c_{\lambda,\mu}(k_2) &\geq 0, \label{Jack Koeffizienten nichtnegativ} \\
x \geq 0 \implies m_{\mu}(x) &\leq m_{\lambda}(x), \label{Muirhead} \\
0 \leq |C_{\lambda}^{(1/k_2)}(x)| &\leq C_{\lambda}^{(1/k_2)}(|x|) \leq \Bigg(\sum\limits_{i=1}^N |x_i|\Bigg)^{|\lambda|}, \label{Ungleichungskette C-Normierung} \\
P_{\lambda}^{(1/k_2)}({\bf 1})
&=
\frac{(k_2 \cdot N)_{\lambda}^{(1/k_2)}}{k_2^{|\lambda|} \cdot c_{\lambda}(1/k_2)}, \label{P(1)} \\
C_{\lambda}^{(1/k_2)}({\bf 1}) 
&\geq 1. \label{C(1) mind. 1}
\end{align}
\end{lemma}

\begin{proof}
For \eqref{Jack Koeffizienten nichtnegativ} see Theorem 1.1 in \cite{KS}.
Moreover, \eqref{Muirhead} is a special case of Muirhead's theorem 2.18 in \cite{HLP}, and
 \eqref{Ungleichungskette C-Normierung} follows from \eqref{norming-p-c}, \eqref{sum-identity-c}, and \eqref{Jack Koeffizienten nichtnegativ}, where the second inequality of \eqref{Ungleichungskette C-Normierung} can also be found in the proof of Lemma 2.1 in \cite{R4}.
 \eqref{P(1)} is given in Table 5 in [DES].
Finally, \eqref{C(1) mind. 1} is shown for instance in the proof of Proposition 1 in \cite{Kan}.
\end{proof}

In the remainder of this section, we prove the following limit result:

\begin{theorem}\label{Limes_1_Jackpolynom_Lemma}
Let $k_2>0$ and $N\in\N$.
Let $(\lambda(k))_{k>0}\subset\Lambda_N$ be a sequence of partitions such that $(\lambda(k)/k)_{k>0}$ is  bounded. 
Then locally uniformly for $x \in C_N^B$,
\begin{equation}\label{Limes_1_Jackpolynom}
(P_{\lambda(k)}^{(1/k_2)}(x))^{1/k} - \prod_{i=1}^{N} x_i^{\lambda(k)_i/k} \to 0 \quad\text{for}\quad k \to \infty.
\end{equation}
\end{theorem}

The proof is based on the following  lemma which is also needed in Section 4 in this form.

\begin{lemma}\label{Vorbereitung_Approximation}
Let $(\lambda(k))_{k>0}\subset\Lambda_2$ and $(b(k))_{k>0} \subset[0,\infty[$
such that $(\lambda(k)/k)_{k>0}$ and $(b(k)/k)_{k>0}$ are bounded.
Furthermore, let $a \in \R\setminus\{0\}$, $d \in \R$, $k_0 > 0$, and $u>0$ with
\begin{equation}\label{Bedingung_1_Vorbereitung_Approximation}
a \cdot \lambda(k)_i+b(k) \geq 0 \quad\text{and}\quad
a \cdot \lambda(k)_i+b(k)+d \geq u \quad\text{for}\quad i=1,2, \; k \geq k_0 .
\end{equation}
If
$M_k:=\frac{1}{k}\ln\bigg(\prod\limits_{j=\lambda(k)_2+1}^{\lambda(k)_1}(a \cdot j + b(k) +d)\bigg)$ and
$$S_k:=\frac{1}{ak}  ((a \lambda(k)_1 + b(k)) \cdot (\ln(a \lambda(k)_1 + b(k)) - 1) - (a \lambda(k)_2 + b(k)) \cdot (\ln(a \lambda(k)_2+b(k))-1)),$$
then  with the  convention $0 \cdot \ln(0)=0$, $M_k-S_k\to0$ for $k\to\infty$.
\end{lemma}

\begin{proof}
We write $x_k \approx y_k$ for  real sequences with  $x_k - y_k \ra 0$ as $k \to \infty$.
We now check the two $\approx $ in the following consideration:
\begin{align}
M_k
&=
\frac{1}{k}  \sum\limits_{j=\lambda(k)_2+1}^{\lambda(k)_1} \ln(a \cdot j +b(k)+d)
\>\>{\approx}\>\>
\frac{1}{k} \int\limits_{\lambda(k)_2}^{\lambda(k)_1} \ln(a \cdot x +b(k)+d)dx \notag\\
&= S_k(d):=
\frac{1}{a k} \cdot ((a \cdot \lambda(k)_1 + b(k) + d) \cdot (\ln(a \cdot \lambda(k)_1 + b(k) + d) - 1)
\notag\\ &\quad\quad\quad\quad\quad\quad  - (a \cdot \lambda(k)_2 + b(k) + d) \cdot (\ln(a \cdot \lambda(k)_2+b(k)+d)-1)) \quad
\approx S_k.\notag
\end{align} 
In fact, the first $\approx $ follows  from the monotonicity of $ \ln(a \cdot x +b(k)+d)$ and from the fact that we may 
 omit and/or add a summand at the beginning and/or end
in the sum in
$$ \frac{1}{k}\sum\limits_{j=\lambda(k)_2+1}^{\lambda(k)_1} \ln(a \cdot j +b(k)+d),$$
such that $\approx$ is preserved by our assumptions.
For the second $\approx$ we put
$r_{k,i}:=a  \lambda(k)_i+b(k)$ for  $i=1,2$).
Then, by the $\Delta$-inequality and the mean value theorem,
\begin{align}
  |a|\cdot k\cdot|S_k(d)-S_k| &\le \sum_{i=1}^2 \Biggl( |r_{k,i}|\cdot |\ln(r_{k,i}+d)-\ln r_{k,i}| + |d|\cdot |\ln(r_{k,i}+d)-1|\Bigr) \notag\\
  &\le \sum_{i=1}^2\max(1, r_{k,i}/(r_{k,i}+d)) +O(\ln k),\notag
\end{align}
which immediately gives the second  $\approx$.
\end{proof}

\begin{proof}[Proof of Theorem \ref{Limes_1_Jackpolynom_Lemma}]
  Let $x \in C_N^B$ and let $c_{\lambda(k),\lambda(k)}(k_2) \coloneqq 1$. Then,  by \eqref{Jack Koeffizienten nichtnegativ}, \eqref{Muirhead}, and
  \eqref{P(1)},
\begin{align}\label{Abschaetzung Jack}
  \prod\limits_{i=1}^N x_i^{\lambda(k)_i}
  &\leq  m_{\lambda(k)}(x) \leq P_{\lambda(k)}^{(1/k_2)}(x) = \sum\limits_{\mu \leq \lambda(k)} c_{\mu,\lambda(k)}(k_2) \cdot m_{\mu}(x) \notag\\
  &\leq m_{\lambda(k)}(x) \cdot \sum\limits_{\mu \leq \lambda(k)} c_{\mu,\lambda(k)}(k_2) 
  = \frac{m_{\lambda(k)}(x) \cdot P^{(1/k_2)}_{\lambda(k)}({\bf 1})}{N!}\notag\\ &= \frac{(k_2 \cdot N)_{\lambda(k)}^{(1/k_2)} \cdot m_{\lambda(k)}(x)}{N! \cdot k_2^{|\lambda(k)|} \cdot c_{\lambda(k)}(1/k_2)} 
  \leq \frac{(k_2 \cdot N)_{\lambda(k)}^{(1/k_2)}}{k_2^{|\lambda(k)|} \cdot c_{\lambda(k)}(1/k_2)} \cdot \prod\limits_{i=1}^N x_i^{\lambda(k)_i}.
\end{align}
We next use $\lambda(k)_{N+1}=0$, the Pochhammer symbol, and \eqref{def-c-lambda} and observe that
\begin{align}
H_k:&=\frac{1}{k} \ln \Bigg(  \frac{(k_2 \cdot N)_{\lambda(k)}^{(1/k_2)}}{k_2^{|\lambda(k)|} \cdot c_{\lambda(k)}(1/k_2)}\Bigg)
= \frac{1}{k}\ln \Bigg( \prod\limits_{(i,j) \in \lambda(k)} \frac{k_2 \cdot (N-i+1) + j - 1}{k_2 \cdot (\lambda(k)_j'-i+1) +\lambda(k)_i-j}  \Bigg) \notag\\
&=  \frac{1}{k} \sum\limits_{i=1}^N \sum\limits_{m=i}^N \Bigg( \ln \Bigg(  \prod\limits_{j=\lambda(k)_{m+1}+1}^{\lambda(k)_m} (k_2 \cdot (N-i+1) + j - 1)  \Bigg)  \notag\\
&\quad\quad\quad\quad\quad\quad - \ln \Bigg( \prod\limits_{j=\lambda(k)_{m+1}+1}^{\lambda(k)_m} (k_2 \cdot (m-i+1) +\lambda(k)_i-j) \Bigg) \Bigg).
 \notag\end{align}
If we apply Lemma \ref{Vorbereitung_Approximation}
with $\lambda(k) = (\lambda(k)_{m+1},\lambda(k)_m)$, $a = \pm 1$,  $b(k)= 0$ and $b(k)= \lambda(k)_i$,
as well as $d = k_2 \cdot (N-i+1)-1$ and $d = k_2 \cdot (m-i+1)$, respectively and  $\lambda(k)_{N+1}=0$, it follows that
\begin{align}
H_k
&\approx \frac{1}{k} \sum\limits_{i=1}^N  \sum\limits_{m=i}^N \Bigg(\lambda(k)_m (\ln(\lambda(k)_m)-1)-\lambda(k)_{m+1} (\ln(\lambda(k)_{m+1})-1) \notag\\
&\>\>\>\>\>\> + (\lambda(k)_i - \lambda(k)_m) (\ln(\lambda(k)_i-\lambda(k)_m)-1)-(\lambda(k)_i-\lambda(k)_{m+1}) (\ln(\lambda(k)_i-\lambda(k)_{m+1})-1)
\Bigg) \notag\\
&= \frac{1}{k}\sum\limits_{i=1}^N \big(\lambda(k)_i \cdot (\ln(\lambda(k)_i)-1) - \lambda(k)_i \cdot (\ln(\lambda(k)_i)-1)\big)
= 0
\end{align}
for $k\to0$. Hence,
\begin{align}\label{Limes P(1)}
\Bigg( \frac{(k_2 \cdot N)_{\lambda(k)}^{(1/k_2)}}{k_2^{|\lambda(k)|} \cdot c_{\lambda(k)}(1/k_2)} \Bigg)^{1/k} \xra[k \to \infty]{} 1.
\end{align}
Finally, this, \eqref{Abschaetzung Jack},  and the local boundedness of $\prod_{i=1}^{N} x_i^{\lambda(k)_i/k}$
immediately lead to the assertion where  the convergence is locally uniform in $x$.
\end{proof}

\section{Proof of Eq.~\eqref{limit-BN-functions-intro1}  via the series representation}

In this section, we prove the limit result \eqref{limit-BN-functions-intro1} in Theorem \ref{limit-BN-functions-intro} for arbitrary
$k_2>0$ by using the hypergeometric series representation 
of $J^B_{(k_1,k_2)}$ in terms of the Jack polynomials.
In fact, according to Subsection 2.2 and Proposition 4.5 in \cite{R4},  $J^{B}_{(k_1,k_2)}$
can be written as the hypergeometric functions
\begin{align}\label{series-rep-bessel}
J^{B}_{(k_1,k_2)}(x,y) 
&= {}_0F_1^{1/k_2}\Big(k_1+(N-1)k_2+\frac{1}{2};\frac{x^2}{2},\frac{y^2}{2}\Big) \\
&= \sum\limits_{\lambda \in \Lambda_N}
\frac
{1}{(k_1+(N-1)k_2+\frac{1}{2})_{\lambda}^{1/k_2} \cdot |\lambda|!} 
\cdot 
\frac{C_{\lambda}^{(1/k_2)}(\frac{x^2}{2}) \cdot C_{\lambda}^{(1/k_2)}(\frac{y^2}{2})}
{C_{\lambda}^{(1/k_2)}({\bf 1})}
\notag\end{align}
for $k_1 \geq 0$, $k_2>0$, and $x,y\in\R^N$, where we use the notation $x^2 \coloneqq (x_1^2,\ldots,x_N^2)$.

Using \eqref{norming-p-c}, \eqref{P(1)}, and the homogeneity of Jack polynomials, we can rewrite this identity as
\begin{align}\label{P-Darstellung Besselfunktion}
J_k^{B}(x,k_1y)
=
\sum\limits_{\lambda \in \Lambda_N}
\frac
{k_1^{2|\lambda|} \cdot c_{\lambda}(1/k_2) \cdot P_{\lambda}^{(1/k_2)}\big(\frac{x^2}{2}\big) \cdot P_{\lambda}^{(1/k_2)}\big(\frac{y^2}{2}\big)}
{(k_2 \cdot N)_{\lambda}^{(1/k_2)} \cdot c'_{\lambda}(1/k_2) \cdot (k_1 + (N-1)k_2 + \frac{1}{2})_{\lambda}^{(1/k_2)}}.
\end{align}
We next derive some lemmas which are needed for the proof of Equation \eqref{limit-BN-functions-intro1}.

\begin{lemma}\label{Kontrolle_Reihenrest}
Let $N\in\N$ and $k_2>0$.
For $c,k_1 \geq 0$ and $x,y \in \R^N$, define 
\begin{align}
R_{k_1,c}(x,y) \coloneqq \sum\limits_{n=\lceil c \cdot k_1 \rceil}^{\infty}\>  \sum\limits_{\lambda \in \Lambda_N :\> |\lambda|=n}
\frac
{k_1^{2n}}{(k_1+(N-1)k_2+\frac{1}{2})_{\lambda}^{1/k_2} \cdot n!} 
\cdot 
\frac{C_{\lambda}^{(1/k_2)}(\frac{x^2}{2}) \cdot C_{\lambda}^{(1/k_2)}(\frac{y^2}{2})}
{C_{\lambda}^{(1/k_2)}({\bf 1})}. \nonumber
\end{align}
Then, locally uniformly in $x,y\in\R^N$,
\begin{align}
\lim\limits_{c \to \infty} \limsup\limits_{k_1 \to \infty} (R_{k_1,c}(x,y))^{1/k_1}=0. \label{Limes_Rest_Reihe}
\end{align}
\end{lemma}

\begin{proof}
As
\begin{align*}
\Big( k_1+(N-1)k_2+\frac{1}{2} \Big)_{\lambda}^{1/k_2}
=
\prod\limits_{i=1}^N \Big( k_1+(N-i)k_2+\frac{1}{2} \Big)_{\lambda_i} \geq k_1^{|\lambda|},
\end{align*}
we obtain from \eqref{sum-identity-c}, \eqref{Ungleichungskette C-Normierung}, \eqref{C(1) mind. 1}, and a classical statement about the Taylor remainder that
\begin{align}
0
\leq
R_{k_1,c}(x,y) 
\leq 
\sum\limits_{n=\lceil c \cdot k_1 \rceil}^{\infty} \frac{\big(k_1 \big| \big| \frac{x^2}{2} \big| \big|_1 \big| \big| \frac{y^2}{2} \big| \big|_1\big)^n}{n!}
\leq
\frac{\exp\big(k_1 \big| \big| \frac{x^2}{2} \big| \big|_1 \big| \big| \frac{y^2}{2} \big| \big|_1\big)}{\lceil c k_1 \rceil !} 
\cdot \Big(k_1 \big| \big| \frac{x^2}{2} \big| \big|_1 \big| \big| \frac{y^2}{2} \big| \big|_1\Big)^{\lceil c k_1 \rceil}. \label{Abschaetzung_Rest_Reihe}
\end{align}
This and the Stirling formula now lead to 
\eqref{Limes_Rest_Reihe}.
\end{proof}

The next lemma is a direct consequence of Lemma \ref{Vorbereitung_Approximation}.

\begin{lemma}\label{Limes_1_ohne_Jackpolynomanteil}
Let $k_2>0$ and $N\in\N$.
Let $(\lambda(k_1))_{k_1 > 0}\subset\Lambda_N$ a family of partitions such that $(c(k_1)\coloneqq \lambda(k_1)/k_1)_{k_1>0}$ is bounded. 
Furthermore, define
\begin{align*}
M_{k_1} &\coloneqq \Bigg( \frac{k_1^{2|\lambda(k_1)|} \cdot c_{\lambda(k_1)}(1/k_2)}{(k_2 \cdot N)_{\lambda(k_1)}^{(1/k_2)} \cdot c'_{\lambda(k_1)}(1/k_2) \cdot (k_1 + (N-1)k_2 + \frac{1}{2})_{\lambda(k_1)}^{(1/k_2)}} \Bigg)^{1/k_1}, \\
S_{k_1} &\coloneqq \prod\limits_{i=1}^N \Bigg( \bigg( \frac{e^2}{(1+c(k_1)_i)c(k_1)_i} \bigg)^{c(k_1)_i} \cdot \frac{1}{1+c(k_1)_i} \Bigg).
\end{align*}
Then $M_{k_1}-S_{k_1} \ra 0$ for $k_1 \to \infty$.

\begin{proof}
The arguments of  the proof of \eqref{Limes P(1)} show that $M_{k_1}/S_{k_1} \ra 1$ for $k_1 \to \infty$. 
Additionally, there is a constant $C>0$ with $1/C \leq S_{k_1} \leq C$ for all $k_1>0$.
This implies the assertion.
\end{proof}
\end{lemma}

We next turn to some approximation which connects functions $f$ on the set  $\Lambda_N$ of partitions with limit functions
$f_0$ on $C_N^B$, where both functions depend on some  parameter in $C_N^B\times C_N^B$, and
where $f$  depends on some scaling parameter $k>0$. In the application of this approximation in the proof of
Eq.~\eqref{limit-BN-functions-intro1}, the function $f_0$ will be almost a continuous function which disappears at infinity, where
``almost'' means that we have some discontinuity on the boundary of $ C_N^B$. In order to incorporate
these boundary cases, we replace the simple $C_0$-condition by the conditions  \eqref{Limeslemma_Reihe_Voraussetzung_Kontrolle_Limes},
\eqref{Limeslemma_Reihe_Voraussetzung_Uebergang_Lambda_durch_k_C}, and \eqref{Limeslemma_Reihe_Voraussetzung_Beschraenktheit_f_0} below.
These  conditions are sufficient for our approximation result and
in the proof of
Eq.~\eqref{limit-BN-functions-intro1}.
As this approximation might be useful also elsewhere, we state it in a  general form:

\begin{lemma}\label{Limeslemma_Reihe}
Let $C \subset \R^N$ and $K \subset \R^M$, and let $\Lambda \coloneqq C \cap \Z^N$.
Consider functions $f:\,]0,\infty[\, \times \Lambda \times K \ra [0,\infty[\,$ and $f_0:C \times K \ra [0,\infty[\,$
            such that for sufficiently large $c_0>0$ the following properties hold:
\begin{align}
 \lim\limits_{k \to \infty} \>\>\sup\limits_{\lambda \in \Lambda: \>||\lambda|| < c_0 \cdot k}\>\> \sup\limits_{x \in K}
|(f(k,\lambda,x))^{1/k}-f_0(\lambda/k,x)|&=0, \label{Limeslemma_Reihe_Voraussetzung_Konvergenz_Summanden} \\
\lim\limits_{c_0 \to \infty} \limsup_{k \to \infty} \sup\limits_{x \in K} \Bigg( \sum\limits_{\substack{\lambda\in\Lambda \\ ||\lambda|| \geq c_0 \cdot k}}
f(k,\lambda,x) \Bigg)^{1/k} &= 0, \label{Limeslemma_Reihe_Voraussetzung_Kontrolle_Restreihe} \\
\lim\limits_{c \in C : \> ||c|| \to \infty} \>\> \sup\limits_{x \in K} f_0(c,x) &=0, \label{Limeslemma_Reihe_Voraussetzung_Kontrolle_Limes} \\
  \lim\limits_{k \to \infty}
\sup\limits_{x \in K} \Bigg|
\sup\limits_{c \in \Lambda/k :\>\> ||c|| < c_0} f_0(c,x)
-
\sup\limits_{c \in C :\>\> ||c|| < c_0} f_0(c,x)
\Bigg|
&=
0, \label{Limeslemma_Reihe_Voraussetzung_Uebergang_Lambda_durch_k_C} \\
\sup\limits_{c \in C} \sup\limits_{x \in K} f_0(c,x) < \infty. \label{Limeslemma_Reihe_Voraussetzung_Beschraenktheit_f_0}
\end{align}
Then,
\begin{align*}
\lim\limits_{k \to \infty} \sup\limits_{x \in K} \Bigg| 
\Bigg( \sum\limits_{\lambda\in\Lambda} f(k,\lambda,x) \Bigg)^{1/k} - \sup\limits_{c \in C} f_0(c,x) \Bigg| = 0.
\end{align*}

\begin{proof}
  Let $\epsilon>0$. Clearly,
we find  $c_0>0$ and $k_0>0$
sufficiently large such that, for all $k \geq k_0$,  the left-hand sides of
\eqref{Limeslemma_Reihe_Voraussetzung_Konvergenz_Summanden},
\eqref{Limeslemma_Reihe_Voraussetzung_Kontrolle_Restreihe},
\eqref{Limeslemma_Reihe_Voraussetzung_Kontrolle_Limes}, and 
\eqref{Limeslemma_Reihe_Voraussetzung_Uebergang_Lambda_durch_k_C}
without the limits $k\to\infty$ and $c_0\to\infty$ are bounded by
 $\frac{\epsilon}{5}$.
We  now claim that  there exists some $k_1>0$ such that for all $ k \geq k_1$ we have
\begin{equation}\label{interim-claim-approx}
\sup\limits_{x \in K} 
\Bigg|
\Bigg( \sum\limits_{\lambda\in\Lambda} f(k,\lambda,x) \Bigg)^{1/k}
-
\sup\limits_{\lambda \in \Lambda ;\> ||\lambda|| < c_0 \cdot k} (f(k,\lambda,x))^{1/k}
\Bigg|
\leq
\frac{2\epsilon}{5}.
\end{equation}
To show this, we conclude from \eqref{Limeslemma_Reihe_Voraussetzung_Kontrolle_Restreihe}, $f \geq 0$, and
$|\{z \in \Z^N \, | \, ||z||_{\infty} < n\}|=(2 \lceil n \rceil-1)^N$ for $n > 0$, that
for $x \in K$, $k \geq k_0$, and a constant $a>0$,
\begin{align}
 \sup\limits_{\lambda \in \Lambda:\> ||\lambda|| < c_0 \cdot k} (f(k,\lambda,x))^{1/k}
&\leq
\Bigg( \sum\limits_{\lambda\in\Lambda} f(k,\lambda,x) \Bigg)^{1/k} \notag \\ \notag
&\leq
(|\{\lambda \in \Lambda \, : \, ||\lambda|| < c_0 \cdot k \} | + 1)^{1/k}
\cdot \\ &\quad\quad\cdot \max\Bigg( 
\sup\limits_{\lambda \in \Lambda : \>\> ||\lambda|| < c_0 \cdot k} (f(k,\lambda,x))^{1/k},
\Bigg( \sum\limits_{\lambda\in\Lambda : \>\> ||\lambda|| \geq c_0 \cdot k} f(k,\lambda,x) \Bigg)^{1/k}
\Bigg) \notag\\
&\leq (a \cdot k^N)^{1/k} \cdot \max\Bigg( \sup\limits_{\lambda \in \Lambda: \>\>||\lambda|| < c_0 \cdot k} (f(k,\lambda,x))^{1/k}, \frac{\epsilon}{5} \Bigg). \label{Vergleich_Supremum_Reihe}
\end{align}
Moreover, we obtain from 
\eqref{Limeslemma_Reihe_Voraussetzung_Konvergenz_Summanden} and
\eqref{Limeslemma_Reihe_Voraussetzung_Beschraenktheit_f_0}, that there exists a $k_2>0$ such that
\begin{align}
\sup\limits_{k \geq k_2} \sup\limits_{\lambda \in \Lambda : \>\> ||\lambda|| < c_0 \cdot k} \sup\limits_{x \in K} (f(k,\lambda,x))^{1/k} < \infty. \label{Beschraenktheit f}
\end{align}
Furthermore as $(a \cdot k^N)^{1/k} \ra 1$ for $k \to \infty$, \eqref{Vergleich_Supremum_Reihe} and \eqref{Beschraenktheit f} imply \eqref{interim-claim-approx}.

We next use \eqref{interim-claim-approx},
\eqref{Limeslemma_Reihe_Voraussetzung_Konvergenz_Summanden},
\eqref{Limeslemma_Reihe_Voraussetzung_Kontrolle_Limes}, 
\eqref{Limeslemma_Reihe_Voraussetzung_Uebergang_Lambda_durch_k_C}, 
$f_0 \geq 0$, 
and the $\Delta$-inequality and observe for $k \geq k_1$ that
\begin{align*}
\limsup\limits_{k \to \infty} \sup\limits_{x \in K} \Bigg| 
\Bigg( \sum\limits_{\lambda\in\Lambda} f(k,\lambda,x) \Bigg)^{1/k} - \sup\limits_{c \in C} f_0(c,x) \Bigg| \leq \epsilon.
\end{align*}
As $\epsilon>0$ was arbitrary, the lemma is proved.
\end{proof}
\end{lemma}

We next consider the following elementary lemma on some function  which may be discontinuous.

\begin{lemma}\label{Maximum g_z}
For $z \geq 0$, consider the function $g_z:[0,\infty[\, \ra \R$ with $g_z(0):=1$ and, for $c>0$,
\begin{align*}
g_z(c) := \bigg( \frac{e^2 \cdot z}{(1+c)c} \bigg)^{c} \cdot \frac{1}{1+c}.
\end{align*}
Then, $g_z$ has a unique global maximum at $c^{max}_z = (\sqrt{1+4z}-1)/2$ with
\begin{align}\label{g_max}
g_z(c^{max}_z)=2 \cdot e^{\sqrt{1+4z}-1} \cdot \frac{1}{\sqrt{1+4z}+1}=e^{\sqrt{1+4z}-1} \cdot \frac{\sqrt{1+4z}-1}{2z}. 
\end{align}
\end{lemma}

\begin{proof} Assume first that $z > 0$. In this case, $g_z$ is continuous on $[0,\infty[$ with
      $\lim_{c\to\infty}g_z(c)=0$. Moreover, by elementary calculus, we have 
      $\frac{d}{dc} \ln(g_z(c))=0$ precisely for $\frac{z}{(1+c)c}=1$,  i.e., for $c = c^{max}_z$ as in the lemma.
      As $\frac{d^2}{dc^2} \ln(g_z(c)) = -\frac{1}{1+c} - \frac{1}{c} < 0$, we conclude that  we have a global
      maximum at $c = c^{max}_z$, and we also get \eqref{g_max}.
      Furthermore, for $z=0$, $g_z$ is discontinuous in $0$ and maximal at $0=c^{max}_0$, and \eqref{g_max} is also valid.
\end{proof}

We now prove Equation \eqref{limit-BN-functions-intro1} in Theorem \ref{limit-BN-functions-intro} for general multiplicities parameters $k_2>0$:

\begin{proof}[Proof of Eq.~\eqref{limit-BN-functions-intro1} in Theorem \ref{limit-BN-functions-intro} for $k_2>0$]
  In  Lemma \ref{Limeslemma_Reihe}
we  take $C \coloneqq C_N^B$, $\Lambda=\Lambda_N$, $M \coloneqq 2N$,
and  $K \subset C_N^B \times C_N^B$ a compactum.
For  $k>0$, $\lambda\in\Lambda_N$, $c \in C_N^B$, and $(x,y) \in K$ we define
\begin{align*}
f(k,\lambda,x,y) 
&\coloneqq
\frac
{k^{2|\lambda|} \cdot c_{\lambda}(1/k_2) \cdot P_{\lambda}^{(1/k_2)}\big(\frac{x^2}{2}\big) \cdot P_{\lambda}^{(1/k_2)}\big(\frac{y^2}{2}\big)}
{(k_2 \cdot N)_{\lambda}^{(1/k_2)} \cdot c'_{\lambda}(1/k_2) \cdot (k + (N-1)k_2 + \frac{1}{2})_{\lambda}^{(1/k_2)}}, \text{ and} \\
f_0(c,x,y) &\coloneqq 
\prod\limits_{i=1}^N g_{(x_i y_i)^2/4}
\end{align*}
with the functions $g_z$ of the preceding lemma.
We now check the conditions of Lemma \ref{Limeslemma_Reihe}.
This lemma and Lemma \ref{Maximum g_z} then lead to the claim \eqref{limit-BN-functions-intro1}.

We have  $f \geq 0$  by the definition of the Jack polynomials and  \eqref{Jack Koeffizienten nichtnegativ}, while 
 $f_0 \geq 0$ is obvious. Moreover, by Lemma \ref{Maximum g_z},
 $z \mapsto g_z(c^{max}_z)$ is continuous, and  $(c,z) \mapsto g_z(c)$ is bounded on $[0, \infty[\, \times [0,u]$ for  $u>0$. 
This leads to \eqref{Limeslemma_Reihe_Voraussetzung_Beschraenktheit_f_0}.
Furthermore, \eqref{Limeslemma_Reihe_Voraussetzung_Konvergenz_Summanden} is a conclusion from Theorem \ref{Limes_1_Jackpolynom_Lemma},
Lemma \ref{Limes_1_ohne_Jackpolynomanteil}, and the fact that  $(c,x) \mapsto \prod_{i=1,\ldots,N} x_i^{c_i}$ is locally bounded and 
 $c \mapsto f_0(c, {\bf 1}, {\bf 1})$  bounded.
 \eqref{Limeslemma_Reihe_Voraussetzung_Kontrolle_Restreihe} follows from \eqref{norming-p-c}, \eqref{P(1)}, and Lemma \ref{Kontrolle_Reihenrest}. Moreover, as $(c,z) \mapsto g_z(c)$ is bounded on $[0, \infty[\, \times [0,u]$ for  $u>0$, 
$g_z(c)$ is increasing in $z$, and as   for  $z \geq 0$, $g_z(c) \ra 0$ for $c \ra \infty$, we obtain  \eqref{Limeslemma_Reihe_Voraussetzung_Kontrolle_Limes}.

Therefore, we only have to prove the critical part \eqref{Limeslemma_Reihe_Voraussetzung_Uebergang_Lambda_durch_k_C}.
For this  let $\epsilon>0$ and define
$$u \coloneqq \max\{(x_i \cdot y_i)^2/4 \, : \, (x,y) \in K, \, i=1,\ldots,N\}.$$
Choose $\delta>0$ such that for all $z \in [0,\delta[\,$, $|g_z(c^{max}_z) - g_z(0)| \leq \epsilon$.
Additionally,  for $k>0$ define $c^{(k)}_z \coloneqq \max\{c \in \N_0/k \, : \, c \leq c^{max}_z\}$.
As  $(c,z) \mapsto g_z(c)$ is uniformly continuous on $[0,c^{max}_u] \times [\delta,u]$, there is a $k_0 > 1$ such that for all $k > k_0$ and $z \in [\delta,u]$ we have 
$|g_z(c^{max}_z) - g_z(c^{(k)}_z)|<\epsilon$. 
Next, define $z:K \ra \R^N$ with $z_i(x,y)=(x_i \cdot y_i)^2/4$ for $i=1, \ldots, N$.
Furthermore, for $k>k_0$, we define  $h^{(k)} : K \to \Lambda/k$ such that for $i=1,\ldots,N$
we put $h^{(k)}_i(x,y)=0$ for  $z_i(x,y) < \delta$,
and $h^{(k)}_i(x,y)=c^{(k)}_{z_i(x,y)}$ for  $z_i(x,y) \geq \delta$. 
Moreover, define $h : K \ra C$ with $h_i(x,y)=c^{max}_{z_i(x,y)}$ for $i=1,\ldots,N$.
Then, for $k>k_0$, $g_{max} \coloneqq \max_{c \geq 0, z \in [0,u]} g_z(c) < \infty$, and $(x,y) \in K$,
\begin{align*}
&\>\>\>\>\>\>\Bigg| \sup\limits_{c \in \Lambda/k} f_0(c,x,y) - \sup\limits_{c \in C} f_0(c,x,y) \Bigg|
\leq |f_0(h^{(k)}(x,y),x,y)-f_0(h(x,y),x,y)| \\
&= \Bigg| \prod\limits_{i=1}^N g_{z_i(x,y)}\Big(c^{(k)}_{z_i(x,y)}\Big) - \prod\limits_{i=1}^N g_{z_i(x,y)}(c^{max}_{z_i(x,y)}) \Bigg| \\
&=\Bigg| \sum\limits_{j=1}^{N} \Bigg( \prod\limits_{i=1}^{j-1} g_{z_i(x,y)}\Big(c^{(k)}_{z_i(x,y)}\Big) \Bigg)
\Bigg( \prod\limits_{i=j+1}^N g_{z_i(x,y)}\Big(c^{max}_{z_i(x,y)}\Big) \Bigg)
\Big( g_{z_j(x,y)}\Big(c^{(k)}_{z_j(x,y)}\Big) - g_{z_j(x,y)}(c^{max}_{z_j(x,y)}) \Big) \Bigg|   \\
&\leq N \cdot g_{max}^{N-1} \cdot \epsilon
\end{align*}
where the second equality follows by a telescoping sum argument.
As $\epsilon>0$ is  arbitrarily small and $h^{(k)} \leq h$ componentwise, we obtain
\eqref{Limeslemma_Reihe_Voraussetzung_Uebergang_Lambda_durch_k_C} 
for $c_0 \geq \sup_{(x,y) \in K}||h(x,y)||_1$ as claimed.
\end{proof}

\section{Application to Bessel processes}

In this section, we  use the limit result \eqref{limit-BN-functions-intro1}
 on the Bessel functions $J_{(k_1,k_2)}^B$
in order to derive the weak limit Theorem \ref{limit-processes-intro}
for Bessel processes.
For this, we first recapitulate from the introduction  and from examples in \cite{V2} that a Bessel process
$(X^{B,w}_{t,k,x})_{t \geq 0}$ of type $B_N$ with multiplicity $k=(k_1,k_2)$, drift $w \in C^N_B$, and starting point $x\in C^N_B$ 
has the transition probability
\begin{equation}\label{transition-b-case}
K_{t,k}^{B,w}(x,A)=\frac{2^N N! \cdot c_k^B e^{-||w||_2^2 t/2}}{t^{\gamma_k^B+N/2}}
\cdot \int\limits_A e^{-(||x||_2^2+||y||_2^2)/(2t)} \cdot \frac{J^B_k(\frac{x}{\sqrt t},\frac{y}{\sqrt{t}}) \cdot J_k^B(y,w)}{J_k^B(x,w)} \cdot w_k^B(y)
       dy
\end{equation}
for $t>0$, a Borel set $A\subset C_N^B$, and $x\in C_N^B$ with
\begin{align*}
w_k^B(y) &= \prod\limits_{i,j=1,\ldots,N, i<j} (y_i^2-y_j^2)^{2 k_2} \cdot \prod\limits_{i=1}^N y_i^{2 k_1}, \quad
\gamma_k^B=k_2 N(N-1)+k_1N, \\ 
c_k^B&=\frac{1}{2^{N(k_1+(N-1)k_2+1/2}} \cdot \prod\limits_{j=1}^N \frac{\Gamma(1+k_2)}{\Gamma(1+jk_2)\Gamma(1/2+k_1+(j-1)k_2)}. \\
\end{align*}
We now consider the processes $(X^{B,k_1w}_{t/k_1,k,x})_{t \geq 0}$.
They are continuous Feller diffusions on $C_N^B$ where by \eqref{transition-b-case}, the transition probabilities
have the form
\begin{equation}\label{trans-tilde-h}
K_{t/k_1,k}^{B,k_1 w}(x,A) 
=\int\limits_A \exp( -k_1 \cdot (  \widetilde{H}_{k,x,w,t}(y) + \widetilde{r}_{k,t} ) ) \cdot V_{k_2}(y) \> dy
\end{equation}
with a suitable constant  $\widetilde{r}_{k,t} \in \R$ (which will be studied below) and 
\begin{align*}
\widetilde{H}_{k,x,w,t}(y) &\coloneqq \sum\limits_{i=1}^N \bigg( \frac{y_i^2}{2t} - \ln(y_i^2)+\frac{x_i^2}{2t}+\frac{w_i^2t}{2}+\ln(2t)-1 \bigg) \\ 
&\>\>\>\>\>\>\>-\frac{1}{k_1} \cdot \Big( \ln\Big( J_k^B \Big( \frac{x}{\sqrt{t}},k_1 \frac{y}{\sqrt{t}} \Big) \Big) + \ln( J_k^B(w,k_1 y)) - \ln( J_k^B(x,k_1 w)) \Big),\notag
\end{align*}
where
$$V_{k_2}(y) \coloneqq \prod\limits_{i,j=1,\ldots,N, i<j} (y_i^2-y_j^2)^{2 k_2}.$$
Using Theorem \ref{limit-BN-functions-intro} we now rewrite     \eqref{trans-tilde-h} as
\begin{equation}\label{K-Schlange}
K_{t/k_1,k}^{B,k_1 w}(x,A) 
=\int_A \exp( -k_1 \cdot (  H_{x,w,t}(y) + r_{k,x,w,t}(y) ) ) \cdot V_{k_2}(y) \> dy
\end{equation}
with some continuous function  $r_{k,x,w,t}$ on $C_N^B$ and
$$H_{x,w,t}(y) \coloneqq \sum\limits_{i=1}^N \bigg( \frac{y_i^2}{2t}-S\Big(\frac{x_iy_i}{t}\Big)-S(w_i y_i)-\ln(y_i^2)+S(x_i w_i)+\frac{x_i^2}{2t}+\frac{w_i^2t}{2}+\ln(2t)-1 \bigg)$$
with
$$S(z)\coloneqq\ln(2) + \sqrt{z^2+1} -1 - \ln(\sqrt{z^2+1}+1)=\ln(g_{z^2/4}(c^{max}_{z^2/4})) \quad\quad (z \in \R) $$
where $g_{z^2/4}$ and $c^{max}_{z^2/4}$ are as in Lemma \ref{Maximum g_z}.
We next analyze $\widetilde{H}_{k,x,w,t}$ and $H_{x,w,t}$
where we expect that the minimizer of $H_{x,w,t}$ is the limit of $X^{B,k_1w}_{t/k_1,k,x}$ for $k_1 \to \infty$.
For this we
follow the approach in the introduction and  solve the ODE \eqref{AWP y-intro} to get a candidate for this minimum.
By the additive structure of $H_{x,w,t}(y)$ it is sufficient to consider  the case $N=1$:

\begin{lemma}\label{DGL y}
For all $w \in \R$ and $x > 0$, the initial value problem
\begin{align*}
y'(t)=\frac{\sqrt{y(t)^2w^2+1}}{y(t)}, \; y(0)=x \label{AWP y}
\end{align*}
has the unique solution $y(t) = y_{x,w,t} \coloneqq \sqrt{x^2 + w^2 t^2 + 2t \sqrt{x^2 w^2 + 1}}$.
Moreover, for $x=0$, it admits exactly two solutions, namely $y(t)=\pm y_{0,w,t}$.
\end{lemma}

\begin{proof} This follows immediately from the method of separation of variables.
\end{proof}

As expected, we now obtain:

\begin{lemma}\label{Minimum H}
Let $N=1$, $x,w \geq 0$ and $t>0$. Then
 $H_{x,w,t}:]0,\infty[\to\mathbb R$ is a strictly convex function which has a unique minimum at  $y_{x,w,t}$ defined in Lemma \ref{DGL y}. 
     Moreover, $H_{x,w,t}(y_{x,w,t})=0$.
     \end{lemma}

\begin{proof}
 Let
 $H:=H_{x,w,t}$,  $\alpha:=(x/t)^2$, $\beta:=w^2$,  $A:=\sqrt{\alpha y^2 +1}$, and $B:=\sqrt{\beta y^2 + 1}$.
We first check that $H$ is strictly convex.
As $\alpha y^2 = (A+1) (A-1)$ and $\beta y^2 = (B+1)(B-1)$, we obtain
$$S'(z)=\frac{z}{\sqrt{z^2+1}} \cdot \bigg( 1 - \frac{1}{\sqrt{z^2+1}+1} \bigg)=\frac{z}{\sqrt{z^2+1}+1}$$
and thus
\begin{equation}\label{H'}
    H'(y)=\frac{1}{y} \cdot \bigg( \frac{y^2}{t} - \frac{\alpha y^2}{A+1} - \frac{\beta y^2}{B+1}-2 \bigg)
= \frac{1}{y} \cdot \bigg(\frac{y^2}{t}-A-B\bigg).\end{equation}
This yields $H''(y)=\frac{1}{t} + \frac{1}{(A+B)y^2}>0$, i.e.,  $H$ is strictly convex.

We next check $H'(y_{x,w,t})=0$ which implies together with the strict convexity that $H$ has a unique minimum at $y_{x,w,t}$.
By \eqref{H'},  $H'(y_{x,w,t})=0$ is equivalent to
$$\frac{y_{x,w,t}^2}{t}=\sqrt{\alpha y_{x,w,t}^2 +1}+\sqrt{\beta y_{x,w,t}^2 + 1}.$$
If we put $R:=\sqrt{x^2w^2+1}$ and use
$y_{x,w,t}^2=x^2+w^2t^2+2t\sqrt{x^2w^2+1}$,
we obtain  that
$$\sqrt{1+\frac{x^2y_{x,w,t}^2}{t^2}}=\frac{x^2+tR}{t},
\quad\quad
\sqrt{1+w^2y_{x,w,t}^2}=w^2t+R.$$
Therefore,
\begin{align}\label{identity-opty}
\sqrt{1+\frac{x^2y_{x,w,t}^2}{t^2}}+\sqrt{1+w^2y_{x,w,t}^2}&=
\frac{x^2+tR}{t}
+w^2t+R\notag\\
&=
\frac{x^2+w^2t^2+2tR}{t}
=\frac{y_{x,w,t}^2}{t},
\end{align}
and thus $H'(y_{x,w,t})=0$ as claimed.

For the final step we use $A,B,R$ as above in the case $y=y_{x,w,t}$
and insert the function $S$ into $H$. We then obtain
$$H(y_{x,w,t})=
\frac{y^2_{x,w,t}+x^2+w^2t^2}{2t}-(A+B)
+\ln\!\left(
\frac{(A+1)(B+1)t}
{y^2_{x,w,t}(R+1)}
\right)+R.$$
Hence, by \eqref{identity-opty},
$$H(y_{x,w,t})=
\frac{x^2+w^2t^2-y^2_{x,w,t}}{2t}+
\ln\left(\frac{(A+1)(B+1)t}{y^2_{x,w,t}(R+1)}\right)+R=
\ln\left(\frac{(A+1)(B+1)t}{y^2_{x,w,t}(R+1)}\right) .$$
Moreover, as
$$A+1=\frac{x^2+t(R+1)}{t},
\quad\quad
B+1=w^2t+R+1,$$
we arrive at
$(A+1)(B+1)t=y^2_{x,w,t}(R+1)$.
Hence $H(y_{x,w,t})=0$ as claimed.
\end{proof}

We next recapitulate some well-known fact:

\begin{lemma}\label{Abschaetzung 2 Besselfunktion B}
For $k_1, k_2 \geq 0$, and $x,y \in \R^N$, 
$1 \leq J_k^B(x,y) \leq \exp\left(||x||_2 \cdot ||y||_2\right).$
\end{lemma}

\begin{proof} The first inequality follows from the facts that the term corresponding to $\lambda=(0,\ldots,0)$ in \eqref{series-rep-bessel} is equal to $1$ and that the other terms are nonnegative. The second inequality follows e.g.~immediately from
 Definition 2.35 and Proposition 2.36 in \cite{R3}.
\end{proof}

This result implies the following estimate for $\widetilde{H}_{k,x,w,t}(y)$:

\begin{lemma}\label{Abschaetzung H}
Let $N \in \N$, $x,w \in \R^N$, $y \in \,]0,\infty[\,^N$, $k_1,k_2 \geq 0$, and $t>0$. Then,
\begin{align*}
\widetilde{H}_{k,x,w,t}(y) 
\geq \frac{||y||_2^2}{4t}-\sum\limits_{i=1}^N \ln(y^2_i)+N \cdot \ln\Big( \frac{2t}{e} \Big) - \frac{3 ||x||_2^2}{2t} - \frac{3 ||w||_2^2 t}{2}
\eqqcolon G_{x,w,t}(y).
\end{align*}
\begin{proof}
Using Lemma \ref{Abschaetzung 2 Besselfunktion B}, we obtain
\begin{align*}
\widetilde{H}_{k,x,w,t}(y) &\geq G_{x,w,t}(y) + \frac{||y||_2^2}{4t} + \frac{2 ||x||_2^2}{t} + 2 ||w||_2^2t - \frac{1}{t} \cdot ||x||_2 \cdot ||y||_2 - ||w||_2 \cdot ||y||_2 \\
&= G_{x,w,t}(y) + \Big( \Big|\Big|\frac{\sqrt{2} \cdot x}{\sqrt{t}}\Big|\Big|_2 - \Big|\Big|\frac{y}{\sqrt{8t}}\Big|\Big|_2 \Big)^2 + \Big( ||\sqrt{2t} \cdot w||_2 - \Big|\Big|\frac{y}{\sqrt{8t}}\Big|\Big|_2 \Big)^2 \\
&\geq G_{x,w,t}(y).
\end{align*}
\end{proof}
\end{lemma}

Note that for any $\epsilon>0$, the logarithm is Lipschitz continuous on $[\epsilon,\infty[\,$.
Hence, for $k_2 \geq 0$, the Stirling formula, Eq. \eqref{limit-BN-functions-intro1} in Theorem \ref{limit-BN-functions-intro}, 
and Lemma \ref{Abschaetzung 2 Besselfunktion B} imply that
locally uniformly w.r.t. $x,w,y \in C_N^B$ and $t>0$,
$\widetilde{r}_{k,t} \to 0$ and $r_{k,x,w,t}(y) \to 0$ for $k_1 \to \infty$.
On the other hand, Lemma \ref{Abschaetzung H} yields the following result for
$\widetilde{H}_{k,x,w,t}(y)$ for $||y|| \ra \infty$:

\begin{corollary}\label{H-Schlange y gross}
Let $N \in \N$. Then uniformly in $k_1,k_2 \geq 0$ and locally uniformly in $x,w \in \R^N$ and $t>0$, 
$\widetilde{H}_{k,x,w,t}(y) \to \infty$ for $y \in \,]0,\infty[\,^N$ with $||y|| \to \infty$.
\end{corollary}

Applying the previous results, we can now derive the main result of this section:

\begin{theorem}\label{schwacher Limes K-Schlange}
Let $N \in \N$,  $k_2 \geq 0$, $x,w \in C_N^B$, and $t >0$. Then, 
\begin{align*}
X_{t/k_1,k,x}^{B,k_1 w} \xrightarrow[k_1 \to \infty]{} y_{x,w,t} \coloneqq (y_{x_1,w_1,t}, \ldots, y_{x_N,w_N,t}) \text{ in probability}
\end{align*}
with $y_{x_i,w_i,t}$ as  in Lemma \ref{DGL y}.
This convergence is locally uniformly in $x,w \in C_N^B$ and $t>0$ in the sense that
 for all $\epsilon>0$ and  compacta $U \subset C_N^B \times C_N^B \times \,]0,\infty[\,$,
\begin{align*}
\sup\limits_{(x,w,t) \in U} P(||X_{t/k_1,k,x}^{B,k_1 w}-y_{x,w,t}|| \geq \epsilon) 
\xrightarrow[k_1 \to \infty]{} 0.
\end{align*}
\begin{proof}
Let $U$ and $\epsilon$ be as in the theorem, and let $\delta>0$.
Then \eqref{trans-tilde-h}, Corollary \ref{H-Schlange y gross}, and the properties of $\widetilde{r}_{k,t}$ imply 
that there exist a compactum $A \subset C_N^B$ and  $k_1' \geq 1$ such that
for  $(x,w,t) \in U$ and $k_1 \geq k_1'$,
\begin{align}
K_{t/k_1,k}^{B,k_1 w}(x,C_N^B \setminus A) \leq \frac{\delta}{2}. \label{Kontrolle K-Schlange y gross}
\end{align}
Furthermore, Lemma \ref{Minimum H} as well as   continuity and compactness arguments show that
\begin{align}
\inf\{H_{x,w,t}(y) \, : \, (x,w,t,y) \in U \times A, \, ||y-y_{x,w,t}|| \geq \epsilon\}>0 \label{H>0}
\end{align}
where we used the convention $\ln(0)=-\infty$.
Therefore \eqref{K-Schlange} and the properties of $r_{k,x,w,t}(y)$ imply that there exists a $\widetilde{k}_1>k_1'$ such that for all $k_1 \geq \widetilde{k}_1$ and $(x,w,t) \in U$, 
\begin{align}
K_{t/k_1,k}^{B,k_1 w}(x,\{y \in A \, : \, ||y-y_{x,w,t}|| \geq \epsilon \}) \leq \frac{\delta}{2}. \label{Kontrolle K-Schlange y entfernt von Minimalstelle}
\end{align}
Finally, combining \eqref{Kontrolle K-Schlange y gross} and \eqref{Kontrolle K-Schlange y entfernt von Minimalstelle},
we obtain that
$K_{t/k_1,k}^{B,k_1 w}(x,\{y \in C_N^B \, : \, ||y-y_{x,w,t}|| \geq \epsilon \}) \leq \delta$
for all $(x,w,t) \in U$ and $k_1 \geq \widetilde{k}_1$.
As $\delta>0$ could be chosen arbitrarily small, the claim is proved.
\end{proof}
\end{theorem}

\begin{proof}[Proof of Theorem \ref{limit-processes-intro}]
For $t=0$, the assertion follows from $X_{0/k_1,k,x}^{B,k_1w}=x=y_{x,w,0}$.
Moreover, for $t>0$, the claim is a corollary of Theorem \ref{schwacher Limes K-Schlange}.
\end{proof}

In the end of this section, we briefly consider
a generalization of Example \ref{1-dim-example} and Corollary \ref{limit-processes-intro-cor-1dim}:

\begin{example} 
\label{N-dim-example}
Let $M \geq N \geq 1$ be integers and $\mathbb F=\mathbb R,\mathbb C,\mathbb H$ with real dimension $d=1,2,4$.
Consider the  real Euclidean space $\mathbb F^{M \times N}$  with the scalar product $\langle A, B \rangle = Re \> tr(A^*B)$
where $^* $ is the usual adjoint.
Let $\Pi_N(\mathbb F)$ be the cone of all $N \times N$ positive semidefinite matrices over $\F$,
$\sigma_N : \Pi_N(\F) \to C_N^B$ be the ordered spectral map,
and $p:\F^{M \times N} \to C_N^B$ the  map with $p(A)=\sqrt{\sigma_N(A^* \cdot A)} \> \forall A \in \F^{M \times N}$ 
where  $\sqrt{x} \coloneqq (\sqrt{x_1},\ldots,\sqrt{x_N})\in C_N^B$ for  $x \in C_N^B$.
Then for all $x \in C_N^B$, we have
$p^{-1}(\{x\}) = \{U \Sigma_x V^* \, | \, U \in U_M(\F), \, V \in U_N(\F)\}$ 
with $\Sigma_x=(diag(x),0)^T \in \F^{M \times N}$; see  e.g. Chapter 3 of \cite{HJ}, and for the quaternionic case  Section 7 of \cite{Z}.
Moreover, by \cite{R4}, the associated spherical functions are given by the Bessel functions $J_{(k_1,k_2)}^B$
with $(k_1,k_2)=((M-N+1) \cdot d/2 - 1/2, d/2)$. More precisely, for $x,y\in C_N^B$, these $k_1,k_2$ and the normalized Haar measures $dU, dV$ on 
$U_M(\F)$ and $ U_N(\F)$ respectively we have
\begin{equation}\label{spherical-matrix}
J_{(k_1,k_2)}(x,y)= \int_{ U_M(\F)}\int_{ U_N(\F)}  e^{ \langle U \Sigma_x V^*,\Sigma_y  \rangle}   \>  dV \> dU.
\end{equation}
This leads to the following connection between Brownian motions $(B_t)_{t\ge0}$ on $\F^{M \times N}$ and Bessel processes
on $C_N^B$ with these  $k_1,k_2$ and with drift; see \cite{AuV, V2}:
We  fix some drift matrix $\Lambda \in \F^{M \times N}$ as well as some $x\in C_N^B$ and consider the
Brownian motion $(B_t+t\Lambda)_{t\ge0}$ on $\F^{M \times N}$ with  drift  $\Lambda$ where we assume that
the Brownian motion has the initial distribution
$$ d\sigma_{M,x,\Lambda} (y) :=\frac{1}{J_{(k_1,k_2)}(x,p(\Lambda))} e^{\langle \Lambda,y  \rangle}  d\sigma_{M,x} (y) \quad\quad (y\in \F^{M \times N})$$
for  the probability measure  $\sigma_{M,x}$ on $p^{-1}(\{x\}) $ which appears as push forward of
$dU\otimes dV\in M^1(U_M(\F)\times U_N(\F))$ under $(U,V)\mapsto U \Sigma_x V^* $.
Note that the measures $\sigma_{M,x,\Lambda}$ are probability measures by \eqref{spherical-matrix}.
Under this initial condition,  $(p(B_t+t\Lambda))_{t\ge0}$ then is a Bessel process on $C_N^B$
with $(k_1,k_2)=((M-N+1) \cdot d/2 - 1/2, d/2)$
with start in $x$ and drift 
$w:=p(\lambda)\in C_N^B$ by \cite{V2}.
Therefore, Theorem \ref{limit-processes-intro} and  Theorem \ref{schwacher Limes K-Schlange} imply:
\end{example}

\begin{corollary}
  Let $x,w\in C_N^B$ and
$t \geq 0$.
For every integer $M \ge N$, choose some matrix $\Lambda_M \in \F^{M \times N}$ with $p(\Lambda_M)=w$
and the probability measure $\sigma_{M,x,((M-N+1) d/2 - 1/2)\Lambda_M}$ on $p^{-1}(\{x\}) \subset \F^{M \times N}$
as well as a Brownian motion $(B_t^M)_{t \geq 0}$ on  $\F^{M \times N}$ with this initial distribution.
Then, for $M \to \infty$, the random variables $p(B_{t/((M-N+1) d/2 - 1/2)}^M+t \Lambda_M)$
tend in probability to $y_{x,w,t}\in C_N^B$ as defined in Theorem \ref{schwacher Limes K-Schlange} 
where the convergence is locally uniform w.r.t. $x,w \in C_N^B$ and $t>0$.
\end{corollary}

\section{One-dimensional Dunkl kernels and processes}

In this section we extend the preceding results to one-dimensional Dunkl kernels and Dunkl processes with drift,
i.e. we consider the case $R=B_1$ and $k=k_1$.
The restriction to $N=1$ is caused by the fact that only in this case a simple explicit integral representation
is available which allows to extend the results in Section 2. We do not know whether series representations
of the Dunkl kernels of type B in terms of non-symmetric Jack polynomials can be used to extend
the results in Sections 3 and 4.
For $N=1$, we start with
the  well-known integral representation
\begin{equation}\label{int-rep-Dunkl-1dim}
  E_k(x,y)=\frac{\Gamma(k+1/2)}{\Gamma(1/2)\Gamma(k)}\int_{-1}^1 e^{xy t} (1-t)^{k-1}(1+t)^k \> dt\quad\quad
   (x,y\in\mathbb C, \>\>k>0)
  \end{equation}
for the Dunkl kernels  $E_k$ from \cite{R1}.
The approach in Section 2 then immediately leads to:

\begin{proposition}\label{limit-Dunkl-kernel-1dim}
  Then, locally uniformly in $x,y\in \mathbb R$,
  \begin{equation}\label{limit-dunklkernel1} 
\lim_{k\to\infty} E_k(x,ky)^{1/k}= e^{\sqrt{ (xy)^2+1}\> -1}\frac{2}{(xy)^2}(\sqrt{( xy)^2+1}\> -1)
  \end{equation}
  and
\begin{equation}\label{limit-dunklkernel2} 
\lim_{k\to\infty} \frac{\frac{d}{dx}E_k(x,ky)}{k\cdot E_k(x,ky)}=\frac{1}{x}(\sqrt{( xy)^2+1}\> -1) .
  \end{equation}
\end{proposition}

\begin{proof} The function $f(t):=  e^{xy t}(1-t^2)$  is nonnegative on $[-1,1]$ with a unique maximum at
  $t_0:=\frac{1}{xy}(\sqrt{( xy)^2+1}\> -1)$ with 
  $$f(t_0)=e^{\sqrt{ (xy)^2+1}\> -1}\frac{2}{(xy)^2}(\sqrt{( xy)^2+1}\> -1).$$
  The proposition now follows from Lemmas \ref{Laplace-method-roots} and \ref{Laplace-method-fraction}.
  \end{proof}

We now recapitulate from 
Definition and Theorem 3.4 in \cite{V2} that
the one-dimensional Dunkl processes $(X^{Dunkl,w}_{t,k,x})_{t \geq 0}$ and hybrid Dunkl-Bessel processes $(X^{DuBe,w}_{t,k,x})_{t \geq 0}$,
with multiplicity $k \geq 0$, drift $w \in \R$, and starting point $x\in\R$, are Feller processes on $\R$.
The first ones have the generators
\begin{align*}
L_k^{Dunkl,w}f=\frac{1}{2} f''(x) + k \cdot \frac{f'(x)}{x} + \frac{\langle \partial_x E_k(x,w),f'(x) \rangle}{E_k(x,w)} 
- \frac{k}{2} \cdot \frac{E_k(-x,w)}{E_k(x,w)} \cdot \frac{f(x)-f(-x)}{x^2}
\end{align*}
and the transition probabilities
\begin{align}
\label{kernels-Dunkl}
K^{Dunkl,w}_{t,k}(x,A)=\frac{e^{-w^2 t/2}}{(2t)^{N \cdot (k + 1/2)} \cdot \Gamma(k + 1/2)}
\cdot \int\limits_A e^{-(x^2+y^2)/(2t)} \cdot \frac{E_k(\frac{x}{\sqrt t},\frac{y}{\sqrt{t}}) \cdot E_k(y,w)}{E_k(x,w)} \cdot y^{2k} dy.
\end{align}
Moreover, with the Bessel functions  $J_k:=J_k^{B_1}$, the second ones have the generators 
\begin{align*}
L_k^{DuBe,w}f=\frac{1}{2} f''(x) + k \cdot \frac{f'(x)}{x} + \frac{\langle \partial_x J_k(x,w),f'(x) \rangle}{J_k(x,w)} 
- \frac{k}{2} \cdot \frac{f(x)-f(-x)}{x^2}
\end{align*}
and the transition probabilities
\begin{align}
\label{kernels-Dunkl2}
K^{DuBe,w}_{t,k}(x,A)=\frac{e^{-w^2 t/2}}{(2t)^{N \cdot (k + 1/2)} \cdot \Gamma(k + 1/2)}
\cdot \int\limits_A e^{-(x^2+y^2)/(2t)} \cdot \frac{E_k(\frac{x}{\sqrt t},\frac{y}{\sqrt{t}}) \cdot J_k(y,w)}{J_k(x,w)} \cdot y^{2k} dy.
\end{align}

Similar to the preceding section, we now study the renormalized processes $(X^{Dunkl,kw}_{t/k,k,x})_{t \geq 0}$ 
and $(X^{DuBe,kw}_{t/k,k,x})_{t \geq 0}$ and the associated transition probabilities $K_{t/k,k}^{Dunkl,kw}(x,A)$ and $K_{t/k,k}^{DuBe,kw}(x,A)$.
Moreover, if an assertion is true in both cases, we  suppress the superscripts ``Dunkl'' and ``DuBe''.
In both cases, the approach in Section 5 yields:

\begin{proposition}\label{schwacher Limes Dunkl Schritt 1}
Let $y_{x,w,t}$ be defined as in Lemma \ref{DGL y}. 
Furthermore, let $\epsilon>0$.
Then, locally uniformly in $x,w \in \R$ and $t \in \,]0,\infty[\,$,
\begin{align*}
P(||X^{kw}_{t/k,k,x}|-y_{x,w,t}| \geq \epsilon) 
\xrightarrow[k \to \infty]{} 0.
\end{align*}
\end{proposition}

For the proof we need the following variant of Lemma \ref{Abschaetzung 2 Besselfunktion B}:

\begin{lemma}\label{Abschaetzung Dunklkern N=1}
For $k \geq 1$ and $x,y \in \R$, 
$\frac{1}{2^{k+1}\sqrt{\pi}} \leq \frac{e^{|xy|/2}-1}{2^{k} \sqrt{\pi} |xy|} \leq E_k(x,y) \leq e^{|xy|}$
where as a convention, the first inequality is an equality for $xy=0$.
\end{lemma}

\begin{proof}
The mean value theorem implies the first inequality.
Furthermore, if $x y \geq 0$, the second inequality follows from \eqref{int-rep-Dunkl-1dim} and
\begin{align*}
E_k(x,y) \geq \frac{1}{2^{k-1}\sqrt{\pi}} \cdot \int\limits_0^{1/2} e^{xyt}dt = \frac{e^{xy/2}-1}{2^{k-1} \sqrt{\pi} xy} \geq \frac{e^{|xy|/2}-1}{2^{k} \sqrt{\pi} |xy|}.
\end{align*}
For $x y \le 0$, the second inequality follows from \eqref{int-rep-Dunkl-1dim} and
\begin{align*}
E_k(x,y) \geq \frac{1}{2^{k}\sqrt{\pi}} \cdot \int\limits_{-1/2}^{0} e^{xyt}dt = \frac{e^{|xy|/2}-1}{2^{k} \sqrt{\pi} |xy|}.
\end{align*}
Finally, for the third inequality, see Proposition 2.36 in \cite{R3}.
\end{proof}

\begin{proof}[Proof of Proposition \ref{schwacher Limes Dunkl Schritt 1}]
Let $F_k=E_k$ in the Dunkl case and $F_k=J_k$ in the hybrid Dunkl-Bessel case.
Moreover, let $H_{x,w,t}:\R \to 0$ be defined as in \eqref{K-Schlange} where $H_{x,w,t}(0)=\infty$ and $N=1$.
Then, analogously to \eqref{K-Schlange}, we rewrite
\begin{align}
&\>\>\>\>\>\>K_{t/k,k}^{kw}(x,A) \notag \\
&=\frac{e^{-kw^2 t/2}}{(2t/k)^{N \cdot (k + 1/2)} \cdot \Gamma(k + 1/2)}
\cdot \int\limits_A e^{-k(x^2+y^2)/(2t)} \cdot \frac{E_k(\frac{x}{\sqrt t},k\frac{y}{\sqrt{t}}) \cdot F_k(y,kw)}{F_k(x,kw)} \cdot y^{2k} dy \notag\\ 
&=\int\limits_A \exp( -k \cdot (  \widetilde{H}^{\star}_{k,x,w,t}(y) + \tilde{r}^{\star}_{k,t} ) ) dy 
=\int\limits_A \exp( -k \cdot ( H_{x,w,t}(y) + r^{\star}_{k,x,w,t}(y) ) ) dy \label{K-Schlange Dunkl}
\end{align} 
with
\begin{align*}
\widetilde{H}^{\star}_{k,x,w,t}(y) &\coloneqq \frac{y^2}{2t} - \ln(y^2)+\frac{x^2}{2t}+\frac{w^2t}{2}+\ln(2t)-1 \\ 
&\>\>\>\>\>\>\>-\frac{1}{k} \cdot \Big( \ln\Big( E_k \Big( \frac{x}{\sqrt{t}},k\frac{y}{\sqrt{t}} \Big) \Big) + \ln( F_k(w,ky)) - \ln( F_k(x,kw)) \Big).
\end{align*}
for suitable $\tilde{r}^{\star}_{k,t} \in \R$ and $r^{\star}_{k,x,w,t} \in C(C_N^B)$.
Then the Stirling formula, Equation \eqref{limit-BN-functions-intro1} in Theorem \ref{limit-BN-functions-intro},
Equation \eqref{limit-dunklkernel1} in Proposition \ref{limit-Dunkl-kernel-1dim}, and Lemma \ref{Abschaetzung Dunklkern N=1} imply
that locally uniformly w.r.t. $x,w \in \R$ and $t>0$, 
also $\tilde{r}^{\star}_{k,t} \to 0$ and $r^{\star}_{k,x,w,t} \to 0$.

Now, let $\delta, \epsilon>0$ and let $U \subset \R \times \R \times \,]0,\infty[\,$ be compact.
We mention that Corollary \ref{H-Schlange y gross} remains valid
after replacing $\widetilde{H}_{k,x,w,t}$, $k_1 \geq 0$, and $N \in \N$ by $\widetilde{H}^{\star}_{k,x,w,t}$, $k \geq 1$, and $N=1$, respectively.
The reason is that we can replace Lemma \ref{Abschaetzung 2 Besselfunktion B} by Lemma \ref{Abschaetzung Dunklkern N=1}
to prove an assertion being similar to Lemma \ref{Abschaetzung H}.
Therefore, analogous to \eqref{Kontrolle K-Schlange y gross}, we obtain a compactum $A \subset \R$ and a $k' \geq 1$ such that
for all $k \geq k'$ and $(x,w,t) \in U$,
$K_{t/k,k}^{k w}(x,\R \setminus A) \leq \delta/2.$
Furthermore, the symmetry of $H_{x,w,t}(y)$ in $y=0$ implies that \eqref{H>0} still holds here
after replacing $||y-y_{x,w,t}||$ by $||y|-y_{x,w,t}|$ and using the new $A$, $U$, and $\epsilon$.
Therefore, analogous to \eqref{Kontrolle K-Schlange y entfernt von Minimalstelle}, we obtain a $\widetilde{k} \geq k'$ 
such that for all $k \geq \widetilde{k}$ and $(x,w,t) \in U$, 
$K_{t/k,k}^{B,k w}(x,\{y \in A \, | \, ||y|-y_{x,w,t}| \geq \epsilon \}) \leq \delta/2.$
Finally, we obtain $K_{t/k,k}^{kw}(x,\{y \in \R \, : \, ||y|-y_{x,w,t}| \geq \epsilon\}) \leq \delta$
which proves the assertion.
\end{proof}

The next aim is to find weights $a_{x,w,t}, b_{x,w,t} \geq 0$ with $a_{x,w,t}+b_{x,w,t}=1$ and
\begin{align}\label{schwacher Limes Dunkl Schritt 2}
K_{t/k,k}^{kw}(x,.) \xrightarrow[k \to \infty]{w} a_{x,w,t} \cdot \delta_{y_{x,w,t}} + b_{x,w,t} \cdot \delta_{-y_{x,w,t}}
\end{align}
for all $x,w \in \R$ and $t \geq 0$ where we use the superscripts ``Dunkl'' and ``DuBe'' in our two cases.
For this, 
we compare $K_{t/k,k}^{kw}(x,\R_+)$ and $K_{t/k,k}^{kw}(x,\R_-)$ for $k \to \infty$:

\begin{proposition}\label{schwacher Limes Dunkl Schritt 3}
Let $h,\widetilde{h}:\R \to \R$ be defined as $h(z)=(\sqrt{z^2+1}-1)/z=z/(\sqrt{z^2+1}+1)$ and $\widetilde{h}(z)=(1+h(z))/(1-h(z))$. 
Let $y_{x,w,t}$ be  as above.
Then, locally uniformly w.r.t.~ $x,w \in \R$ and $t>0$,
\begin{align}
\frac{P(X^{Dunkl,kw}_{t/k,k,x} \geq 0)}{P(X^{Dunkl,kw}_{t/k,k,x} \leq 0)} 
&\xrightarrow[k \to \infty]{} \widetilde{h}(x y_{x,w,t} / t) \widetilde{h}(w y_{x,w,t}) \eqqcolon q^{Dunkl}_{x,w,t}, 
\label{q Dunkl} \\
\frac{P(X^{DuBe,kw}_{t/k,k,x} \geq 0)}{P(X^{DuBe,kw}_{t/k,k,x} \leq 0)} 
&\xrightarrow[k \to \infty]{} \widetilde{h}(x y_{x,w,t} / t) \eqqcolon q^{DuBe}_{x,w,t}.
\label{q DuBe}
\end{align}
\begin{proof}
Applying \eqref{Integralformel 1D Besselfunktion}, \eqref{int-rep-Dunkl-1dim}, and \eqref{K-Schlange Dunkl},
we can rewrite
\begin{align*}
\frac{K_{t/k,k}^{Dunkl,kw}(x,\R_+)}{K_{t/k,k}^{Dunkl,kw}(x,\R_-)}
&= \frac
{\int\limits_0^{\infty} \int\limits_{-1}^1 \int\limits_{-1}^1 f(x,w,t,y,u,v)^k \widetilde{f}(u,v) \>dv \>du \>dy}
{\int\limits_0^{\infty} \int\limits_{-1}^1 \int\limits_{-1}^1 f(x,w,t,y,u,v)^k f^{\star}(u,v) \>dv \>du \>dy}
\end{align*}
with functions $f: \R \times \R \times]0,\infty[ \times [0,\infty[ \times [-1,1] \times [-1,1] \to \R$ and $\widetilde{f}, f^{\star}:[-1,1]^2 \to \R$ defined as
\begin{align*}
f(x,w,t,y,u,v)&=e^{-(x^2+y^2)/(2t)} y^2 \cdot e^{xyu/t} (1-u^2) \cdot e^{wyv} (1-v^2),  \\
\widetilde{f}^{Dunkl}(u,v)&=(1-u)^{-1} (1-v)^{-1}, \; \mkern11.6mu \widetilde{f}^{DuBe}(u,v)=(1-u)^{-1} (1-v^2)^{-1}, \\
f^{\star,Dunkl}(u,v)&=(1+u)^{-1} (1+v)^{-1}, \; f^{\star,DuBe}(u,v)=(1+u)^{-1} (1-v^2)^{-1}.
\end{align*}
Now, the arguments in the proofs of Lemma \ref{Minimum H} and Proposition \ref{limit-Dunkl-kernel-1dim} imply 
that $f(x,w,t,.,.,.)$ has a unique maximum in $(y,u,v)=(y_{x,w,t},h(xy_{x,w,t}/t),h(wy_{x,w,t}))$.
Therefore Lemma \ref{Laplace-method-fraction} and Proposition \ref{schwacher Limes Dunkl Schritt 1} immediately lead to the claim.
\end{proof}
\end{proposition}

 For  our final result,
 we recall that the weak convergence on $\mathbb R$ is metrized by the Prohorov metric $d_P:M_1(\R) \times M_1(\R) \to [0,\infty]$;
$d_P(\mu,\nu)=\inf\{\epsilon>0 \, : \, \mu(B) \leq \nu(B^{\epsilon}) + \epsilon \>\> \forall B \in \mathcal{B}(\R) \}$
with $B^{\epsilon} \coloneqq \{x \in \R \, : \, d(x,B) < \epsilon \}$ and the Borel-$\sigma$-algebra $\mathcal{B}(\R)$.
See e.g. Chapter 13 in \cite{Kl}.

\begin{theorem}
Let $y_{x,w,t}$ be defined as in Lemma \ref{DGL y} and
 $q_{x,w,t}$  as in \eqref{q Dunkl} or \eqref{q DuBe} respectively.
Then, the weak limit assertion \eqref{schwacher Limes Dunkl Schritt 2} holds locally uniformly w.r.t. $x,w \in \R$ and $t>0$ for
$a_{x,w,t}:=q_{x,w,t}/(1+q_{x,w,t})$ and $b_{x,w,t}:=1/(1+q_{x,w,t})$.
This means that for all compacta $U \subset \R \times \R \times ]0,\infty[$,
\begin{align}
\sup\limits_{(x,w,t) \in U} d_P(K_{t/k,k}^{k w}(x,.),a_{x,w,t} \cdot \delta_{y_{x,w,t}} + b_{x,w,t} \cdot \delta_{-y_{x,w,t}})
\xrightarrow[k \to \infty]{} 0. \label{schwacher Limes Dunkl Schritt 6}
\end{align}
\end{theorem}

\begin{proof}
Let $U$ be as in the assumptions.
 Proposition \ref{schwacher Limes Dunkl Schritt 3} and the Lipschitz continuity of the functions $z \mapsto z/(1+z)$ and $z \mapsto 1/(1+z)$ imply 
that
\begin{align}
\sup\limits_{(x,w,t) \in U} |K_{t/k,k}^{kw}(x,\R_+)-a_{x,w,t}| 
\xrightarrow[k \to \infty]{} 0, 
\>\>\>\>\>
\sup\limits_{(x,w,t) \in U} |K_{t/k,k}^{kw}(x,\R_-)-b_{x,w,t}| 
\xrightarrow[k \to \infty]{} 0. 
\label{schwacher Limes Dunkl Schritt 4}
\end{align}
Now, let $\epsilon>0$. Then, Proposition \ref{schwacher Limes Dunkl Schritt 1} and \eqref{schwacher Limes Dunkl Schritt 4} yield
\begin{align}
\sup\limits_{(x,w,t) \in U} |K_{t/k,k}^{kw}(x,\{y \in \R \, : \, |y-y_{x,w,t}| < \epsilon\}) - a_{x,w,t}|
&\xrightarrow[k \to \infty]{} 0, 
\label{schwacher Limes Dunkl Schritt 5.1} \\
\sup\limits_{(x,w,t) \in U} |K_{t/k,k}^{kw}(x,\{y \in \R \, : \, |y+y_{x,w,t}| < \epsilon\})  - b_{x,w,t}|
&\xrightarrow[k \to \infty]{} 0.
\label{schwacher Limes Dunkl Schritt 5.2}
\end{align}
Furthermore, Proposition \ref{schwacher Limes Dunkl Schritt 1}, \eqref{schwacher Limes Dunkl Schritt 5.1}, \eqref{schwacher Limes Dunkl Schritt 5.2}, 
and elementary calculus show that
\begin{align*}
\limsup\limits_{k \to \infty} \sup\limits_{(x,w,t) \in U} d_P(K_{t/k,k}^{k w}(x,.),a_{x,w,t} \cdot \delta_{y_{x,w,t}} + b_{x,w,t} \cdot \delta_{-y_{x,w,t}}) \leq \epsilon.
\end{align*}
As $\epsilon$ can be chosen arbitrarily small, the claim is proved.
\end{proof}

\section*{Declaration of the use of artificial intelligence} During the preparation for this work, the authors used the ChatGPT4 and 5 model families in order
to search for literature and to perform calculations in order to  verify some mathematical statements and to
optimize some proofs.
The authors verified the literature and corrected, and edited these  calculations and  proofs,
and they formulated all  statements and proofs in own words.
They take full responsibility for this work.

\end{document}